\documentclass{amsart}
\usepackage[margin=1in]{geometry}
\usepackage{amsmath}
\usepackage{amssymb}
\usepackage{amsthm}
\usepackage{amscd}
\usepackage{enumerate}
\newcounter{pcounter}

\newcommand{\ZZ}{\Bbb Z}
\newcommand{\RR}{\Bbb R}
\newcommand{\TT}{\Bbb T}
\newcommand{\NN}{\Bbb N}

\newcommand{\CC}{\Bbb C}
\newcommand{\ip}[1]{\langle #1 \rangle}

\newcommand{\varespilon}{\varepsilon}
\newcommand{\vaerpsilon}{\varepsilon}

\newcommand{\actson}{\curvearrowright}

\newtheorem{?}{Question}
\newtheorem{theorem}{Theorem}
\newtheorem{definition}[theorem]{Definition}
\newtheorem{proposition}[theorem]{Proposition}
\newtheorem{cor}[theorem]{Corollary}
\newtheorem{lemma}[theorem]{Lemma}

\DeclareMathOperator{\Int}{int}

\DeclareMathOperator{\vol}{vol}

\DeclareMathOperator{\Sym}{Sym}

\DeclareMathOperator{\Span}{Span}

\DeclareMathOperator{\Map}{Map}
\DeclareMathOperator{\im}{im}

\DeclareMathOperator{\id}{Id}

\DeclareMathOperator{\Aut}{Aut}

\DeclareMathOperator{\tr}{tr}

\DeclareMathOperator{\proj}{proj}

\DeclareMathOperator{\Hom}{Hom}

\DeclareMathOperator{\Tr}{Tr}

\DeclareMathOperator{\CO}{CO}
\DeclareMathOperator{\AP}{AP}
\DeclareMathOperator{\Prob}{Prob}
\DeclareMathOperator{\Ball}{Ball}

\numberwithin{theorem}{section}

\begin{document}
\title{Polish Models and Sofic Entropy}      
\author{Ben Hayes}
\address{Stevenson Center\\
         Nashville, TN 37240}
\email{benjamin.r.hayes@vanderbilt.edu}
\date{\today}

\begin{abstract}

We deduce properties of the Koopman representation of a positive entropy probability measure-preserving action of a countable, discrete, sofic group. Our main result may be regarded as a ``representation-theoretic" version of Sina\v\i's Factor Theorem. We  show that probability measure-preserving actions with completely positive entropy of an infinite sofic group must be mixing and, if the group is nonamenable, have spectral gap. This implies that if $\Gamma$ is a nonamenable group and $\Gamma\actson (X,\mu)$ is  a probability measure-preserving action which is not strongly ergodic, then no action orbit equivalent to $\Gamma\actson (X,\mu)$ has completely positive entropy.  Crucial to these results is a formula for entropy in the presence of a Polish, but a priori noncompact, model.

%

\end{abstract}


\maketitle
\tableofcontents

\section{Introduction}
	This paper is concerned with the structural consequences of positive entropy of probability measure-preserving actions of groups.  Measure-theoretic entropy for actions of $\ZZ$ is classical and goes back to the work of Kolmogorov and Sina\v\i. Roughly speaking, it measures the randomness of the action of $\ZZ.$  It was realized by Kieffer in \cite{Kieff} that one could replace $\ZZ$ with the weaker condition of amenability of the group. Amenability requires a sequence of non-empty finite subsets of the group over which one can average in an approximately translation invariant way. Abelian groups, nilpotent groups and solvable groups are all amenable, whereas the free group on $r$ letters is not if $r\geq 1.$ Entropy for amenable groups is well established as a useful quantity in ergodic theory: it can be computed in many interesting cases (although it is not easy in general), when it is positive it reveals interesting structure on the action, and it has some useful general and intuitive properties.
	
		The most fundamental examples in entropy theory are Bernoulli shifts. If $\Gamma$ is a countable, discrete group and $(\Omega,\omega)$ is a standard probability space the \emph{Bernoulli shift} over $(\Omega,\omega)$ is the probability measure-preserving action  $\Gamma\actson (\Omega,\omega)^{\Gamma}$ defined by
\[(gx)(h)=x(g^{-1}h),\mbox{ $x\in\Omega^{\Gamma},g,h\in\Gamma$}.\]
Bernoulli shifts over amenable groups are completely classified by their entropy. If $\Gamma$ is an infinite, amenable group and $\Gamma\actson (X,\mu)$ is a free, ergodic, probability measure-preserving action  with positive entropy, then $\Gamma\actson (X,\mu)$ factors onto a Bernoulli shift. In fact, in this situation, $\Gamma\actson (X,\mu)$ factors onto any Bernoulli shift of entropy at most that of $\Gamma\actson (X,\mu).$ This is known as Sina\v\i's factor theorem and it was proved for $\Gamma=\ZZ$ by Sina\v\i\ in \cite{Sinai} and for general amenable groups by Ornstein-Weiss in \cite{OrnWeiss}. Sina\v\i's factor theorem is a fundamental result in entropy theory, as it shows that Bernoulli factors ``capture'' all the entropy of a probability measure-preserving action of an amenable group and in some sense shows that entropy is simply a measure of the amount of ``Bernoulli-like'' behavior an action has.

	  In groundbreaking work, L. Bowen defined entropy for probability measure-preserving actions of sofic groups  assuming the existence of a finite generating partition (see \cite{Bow}). This assumption was then removed by Kerr-Li who also defined entropy for actions of sofic groups on compact, metrizable spaces (see \cite{KLi}). We refer the reader to see Section \ref{S:sofic} for the precise definition. Sofic groups form a vastly larger class of groups than amenable groups, it is known that amenable groups, free groups, residually finite groups and linear groups are all sofic and soficity is closed under free products with amalgamation over amenable subgroups (see \cite{DKP},\cite{KerrDykemaPichot2},\cite{ESZ1},\cite{LPaun},\cite{PoppArg}). Thus sofic entropy is a considerable extension of entropy for actions of an amenable group defined by Kieffer.  Roughly, a group is sofic if it has ``almost actions" on finite sets which are ``almost free." A sequence of such ``almost actions'' is called a sofic approximation. Entropy for a probability measure-preserving action $\Gamma\actson (X,\mu)$ of a sofic group is then defined as the exponential growth rate of the number of finitary models there are of the action which are compatible with the fixed sofic approximation. Using sofic entropy, Bowen showed that if two Bernoulli shifts over a sofic group are isomorphic, then their base entropies are the same. Kerr-Li reproved this result with a more direct proof when the base space has infinite entropy in \cite{KLiBern}.
	
		  Since the subject is fairly young there are relatively few known structural consequences of positive \emph{measure} entropy for actions of  \emph{arbitrary} sofic groups. For example, Sina\v\i's factor theorem is not known for sofic groups. Previous consequences of entropy for actions of sofic groups have  either been for topological actions or for specific groups. In \cite{KerrLi2}, Kerr-Li prove that actions with positive \emph{topological} entropy must exhibit some chaotic behavior (for example, they must be Li-Yorke chaotic).  When the acting group is a free group one can consider another form of measure entropy, defined by Bowen in \cite{Bowenfinvariant}, called  $f$-invariant entropy. The $f$-invariant entropy is roughly a ``randomized'' version of sofic entropy.  Interesting consequences have been given by Seward in \cite{SewardFree} for the case of $f$-invariant entropy, but those only apply when the group is a free group.  After the appearance of our preprint, Meyerovitch  showed in \cite{Meyerovitch} that positive sofic entropy implies that almost every stabilizer of the action is finite.

	  In this paper, we will deduce structural consequences of positive measure entropy for actions of arbitrary sofic groups. Our applications are to \emph{spectral} properties of such actions. To the best of our knowledge, aside from the results in our paper, the results of Meyerovitch are the only ones which deduce properties of an action of a \emph{general} sofic group assuming only that the action has positive entropy. Recall that if $\Gamma\actson (X,\mu)$ is a probability measure-preserving action of a countable discrete group $\Gamma,$ we have an induced unitary representation $\rho_{\Gamma\actson (X,\mu)}\colon\Gamma\to \mathcal{U}(L^{2}(X,\mu))$ given by
\[(\rho_{\Gamma\actson (X,\mu)}(g)\xi)(x)=\xi(g^{-1}x).\]
We use $\rho^{0}_{\Gamma\actson (X,\mu)}$ for the restriction of $\rho_{\Gamma\actson (X,\mu)}$ to $L^{2}(X,\mu)\ominus \CC1.$   The representation $\rho^{0}_{\Gamma\actson (X,\mu)}$ is called the \emph{Koopman representation} of $\Gamma\actson (X,\mu).$ Properties of a probability-measure preserving action which only depend upon the Koopman representation are called \emph{spectral} properties. The Koopman representation has played a significant role in ergodic theory since the early days of the subject, being the means to deduce von Neumann's Mean Ergodic Theorem which the Ergodic Theorem relies upon as a first step. Additionally, many other fundamental properties such as compactness, weak mixing, mixing and ergodicity are all spectral properties.
	
	Our results show that one canonical representation of a group plays a special role in entropy theory. Recall that the left regular representation of a group $\lambda\colon \Gamma\to \mathcal{U}(\ell^{2}(\Gamma))$ is defined by $(\lambda(g)f)(x)=f(g^{-1}x).$ We will sometimes use $\lambda_{\Gamma}$ to specify the group. We will also use $\lambda_{\Gamma,\RR}$ (or $\lambda_{\RR}$ if the group is clear) for the orthogonal representation which is the restriction of $\lambda_{\Gamma}$ to $\ell^{2}(\Gamma,\RR).$
	
	Let $\rho_{j}\colon \Gamma\to \mathcal{U}(\mathcal{H}_{j})$ be two unitary representations of a countable discrete group $\Gamma.$ We say that $\rho_{1},\rho_{2}$ are \emph{singular} and write $\rho_{1}\perp \rho_{2}$ if no nonzero subrepresentation of $\rho_{1}$ embeds into $\rho_{2}.$ By subrepresentation, we mean a restriction of $\rho_{j},j=1,2$ to a closed, $\Gamma$-invariant, linear subspace. As is customary, we will often call a closed, $\Gamma$-invariant, linear subspace a subrepresentation as well. The terminology of singularity comes from the case $\Gamma=\ZZ.$ If $\mu$ is a Borel measure on $\TT=\RR/\ZZ,$ then we have a natural unitary representation
	\[\rho_{\mu}\colon\ZZ\to \mathcal{U}(L^{2}(\TT,\mu))\]
given by
\[(\rho_{\mu}(n)\xi)(\theta)=e^{2\pi i n\theta}\xi(\theta).\]
It is easy to check that $\rho_{\mu}\perp \rho_{\nu}$ if and only if $\mu\perp\nu.$  Similar analysis can be done for any abelian group (replacing  $\TT$ by the Pontryagin dual). Thus, singularity of representations is a natural generalization to noncommutative groups of singularity of measures.

	If $\Gamma\actson (X,\mu)$ is a probability measure-preserving action of a countable, discrete sofic group and $\Sigma$ is a sofic approximation of $\Gamma$ (see Definition \ref{S:soficdefn} for the precise definition of a sofic approximation) we use $h_{\Sigma,\mu}(X,\Gamma)$ for the entropy of $\Gamma\actson (X,\mu)$ with respect to $\Sigma$ as defined by Bowen, Kerr-Li.

\begin{theorem}\label{T:Koopmanthingsetc} Let $\Gamma$ be a countably infinite discrete sofic group with sofic approximation $\Sigma.$ Let $\Gamma\actson (X,\mathcal{M},\mu)$ be a measure-preserving action where $(X,\mathcal{M},\mu)$ is a standard probability space. Suppose that $\mathcal{H}\subseteq L^{2}(X,\mu)$ is a subrepresentation such that $\mathcal{M}$ is generated (up to sets of measure zero) by
\[\{f^{-1}(A):f\in \mathcal{H},\mbox{$A\subseteq \CC$ is Borel}\}.\]
If $\rho^{0}_{\Gamma\actson (X,\mu)}\big|_{\mathcal{H}}$  is singular with respect to the left regular representation, then
\[h_{\Sigma,\mu}(X,\Gamma)\leq 0.\]
\end{theorem}

It is well-known that if $(\Omega,\omega)$ is a standard probability space, then for the Bernoulli action $\Gamma\actson (\Omega,\omega)^{\Gamma}$ we have
\begin{equation}\label{E:introkoopmanrepblah}
\rho^{0}_{\Gamma\actson (\Omega,\omega)^{\Gamma}}\cong \lambda_{\Gamma}^{\oplus \infty}.
\end{equation}
 We mentioned before that Sina\v\i's factor theorem is not known for sofic groups. Note that if $\Gamma\actson (X,\mu)$ factors onto a  Bernoulli shift, then by (\ref{E:introkoopmanrepblah}) we have that $\lambda_{\Gamma}^{\oplus \infty}$ embeds into $\rho^{0}_{\Gamma\actson (X,\mu)}.$  In this manner,  Theorem \ref{T:Koopmanthingsetc} may be regarded as a weak version of Sina\v\i's Factor Theorem for sofic groups. It shows that at the representation-theoretic level, an action of a sofic group with positive entropy must contain a subrepresentation of the Koopman representation of a Bernoulli shift. We can thus think of this theorem as a ``representation-theoretic version'' of Sina\v\i's Factor Theorem. It is also the first result which indicates that positive entropy actions of sofic groups must behave in a manner ``similar'' to Bernoulli shifts.

	We can say even more than Theorem \ref{T:Koopmanthingsetc} if we assume a stronger version of positive entropy. Recall that an action $\Gamma\actson (X,\mu)$ has \emph{completely positive entropy} with respect to a sofic approximation $\Sigma$ if whenever $\Gamma\actson (Y,\nu)$ is a factor of $\Gamma\actson (X,\mu)$ and $(Y,\nu)$ is not a one-atom space then $h_{\Sigma,\nu}(Y,\Gamma)>0.$ The following is easy from Theorem \ref{T:Koopmanthingsetc}.
\begin{cor}\label{C:WeakSInaiINtro} Let $\Gamma$ be a countable discrete sofic group with sofic approximation $\Sigma.$ Suppose that $\Gamma\actson (X,\mu)$ has completely positive entropy with respect to $\Sigma$. Then the Koopman representation of $\Gamma\actson (X,\mu)$ is embeddable into the infinite direct sum of the left regular representation.
\end{cor}

Corollary \ref{C:WeakSInaiINtro} was proved for $\Gamma$ amenable by Dooley and Golodets in \cite{Dooley}. Here we should mention that Dooley-Golodets actually prove that if $\Gamma\actson (X,\mu)$ has completely positive entropy and $\Gamma$ is amenable, then the Koopman representation of $\Gamma\actson (X,\mu)$ is isomorphic to an infinite direct sum of the left regular representation. This is an easy consequence of Corollary \ref{C:WeakSInaiINtro} and Sina\v\i's factor theorem. Since Sina\v\i's factor theorem is not known for sofic groups, we do not know if $\rho^{0}_{\Gamma\actson (X,\mu)}\cong \lambda^{\oplus \infty}$ (instead of just $\rho^{0}_{\Gamma\actson (X,\mu)}$ embeds into $\lambda^{\oplus \infty}$) when $\Gamma$ is sofic.

From Corollary \ref{C:WeakSInaiINtro}, we automatically deduce other important structural properties of completely positive entropy actions of a sofic group, specifically mixing and spectral gap.   A probability measure-preserving action $\Gamma\actson (X,\mu)$ is said to:

\begin{enumerate}[(i):]
\item be mixing if $\lim_{g\to\infty}\mu(gA\cap B)=\mu(A)\mu(B)$\mbox{ for all measurable $A,B\subseteq X$,}\\
\item be strongly ergodic if for every sequence $A_{n}$ of measurable subsets of $X$ with $\mu(gA_{n}\Delta A_{n})\to 0$ for all $g\in \Gamma$ we have that $\mu(A_{n})(1-\mu(A_{n}))\to 0,$\\
\item have spectral gap if for every sequence $\xi_{n}\in L^{2}(X,\mu)$ with $\|g\xi_{n}-\xi_{n}\|_{2}\to 0$ for all $g\in\Gamma$, it is true that $\|\xi_{n}\|_{2}\to 0.$
\end{enumerate}
It is easy to see that spectral gap implies strong ergodicity.

%

\begin{cor}Let $\Gamma$ be a countable discrete sofic group with sofic approximation $\Sigma.$ Suppose that $\Gamma\actson (X,\mu)$ has completely positive entropy with respect to $\Sigma$.

(i): If $\Gamma$ is infinite, then $\Gamma\actson (X,\mu)$ is mixing.

(ii): If $\Lambda$ is any nonamenable subgroup of $\Gamma,$ then $\Lambda\actson (X,\mu)$ is strongly ergodic (in fact it has spectral gap).
\end{cor}

	There is another approach to nonamenable entropy, called Rokhlin entropy, due to Seward  in \cite{SewardKrieger}. It has the advantage of being  easy to define and being defined for all groups, but the disadvantage of being extremely difficult to compute. Rokhlin entropy is an upper bound for sofic entropy, but there are no known cases where one can prove an action has positive Rokhlin entropy without using that is has positive sofic entropy. We mention that completely positive Rokhlin entropy can be defined in a similar manner. After the appearance of our preprint, Alpeev in \cite{AALP} proved that actions with completely positive Rohklin entropy are weakly mixing. His approach is completely elementary. It is easy to see that completely positive sofic entropy implies completely positive Rokhlin entropy. Whether or not  actions of a nonamenable group with completely positive Rohklin entropy are mixing, strongly ergodic, or have spectral gap all appear to be open. It seems to be very difficult to deduce properties of the Koopman representation of an action from the assumption that the action has positive Rokhlin entropy.

	Part $(ii)$ of the above Corollary is rather special to nonamenable groups. Recall that two probability measure-preserving actions $\Gamma\actson (X,\mu),\Lambda\actson (Y,\nu)$ of countable discrete groups $\Gamma,\Lambda$ are said to be \emph{orbit equivalent} if there is a measure space isomorphism $\theta\colon (X,\mu)\to (Y,\nu)$ so that $\theta$ takes the $\Gamma$-orbits to $\Lambda$-orbits, i.e. $\theta(\Gamma x)=\Lambda\theta(x)$ for almost every $x\in X.$ Orbit equivalence theory is an area of much current interest relating to operator algebras, ergodic theory and group theory.  Strong ergodicity is an invariant of the orbit equivalence class of the action. Our corollary shows that if $\Gamma\actson (X,\mu)$ is a probability measure-preserving action of a nonamenable group and if the action is not strongly ergodic, then no action orbit equivalent to $\Gamma\actson (X,\mu)$ has completely positive entropy.   Our result indicates that entropy for actions of nonamenable groups may have nontrivial consequences for orbit equivalence theory. It is a celebrated, and deep, fact that all ergodic actions of an amenable group are orbit equivalent. This is due to Connes-Feldman-Weiss and Ornstein-Weiss (see \cite{CFW},\cite{OWOE}). Thus entropy for actions of amenable groups cannot have \emph{any} consequences for orbit equivalence theory.
	
	Spectral gap and strong ergodicity are important properties with many applications. Spectral gap has connections to expander graphs (see \cite{BG1},\cite{BG2},\cite{BourgainY}, \cite{LPS}), orbit equivalence rigidity (see \cite{IoanaSG},\cite{PopaSG}) and  number theory (see \cite{Selberg}). Spectral gap is also related to the Banach-Ruziewicz problem which asks if Lebesgue measure is the unique, finitely-additive, rotation-invariant, probability measure on the sphere defined on all Lebesgue measurable sets (solved independently by Margulis \cite{MargulisRuz} and Sullivan \cite{DennisRuz}). It  is well known that \emph{no} action of an amenable group is strongly ergodic. The consequences  entropy has for strong ergodicity, spectral gap and orbit equivalence indicate that entropy for nonamenable groups may be used to deduce phenomena  not present in the realm of ergodic theory of amenable groups. This reveals  the importance of generalizing entropy to actions of nonamenable groups.

%

    We briefly outline the key differences in our approach and the approach of Dooley-Golodets to prove Corollary \ref{C:WeakSInaiINtro}. Dooley-Golodets first prove Theorem \ref{T:Koopmanthingsetc} when $\Gamma=\ZZ.$ Fourier analysis reduces Theorem \ref{T:Koopmanthingsetc} for $\Gamma=\ZZ$  to the fact that if $\mu,\nu$ are mutually singular, Borel, probability measures on the circle, then there is a continuous function $f$ on the circle with $0\leq f\leq 1$ so that $f$ is ``close'' to the constant function $1$ in $L^{2}(\mu),$ but is ``close'' to zero in $L^{2}(\nu).$ This fact is a simple exercise in measure theory. Dooley-Golodets then deduce Corollary \ref{C:WeakSInaiINtro} for $\Gamma$ amenable from the case $\Gamma=\ZZ$ by using that all amenable groups are orbit equivalent to the integers.

	 Our approach is to prove Theorem  \ref{T:Koopmanthingsetc} for a general sofic group, by simply replacing the harmonic analysis for $\Gamma=\ZZ$ with noncommutative harmonic analysis for a general group. The assumption about singularity of measures is then replaced by singularity of representations of groups. The representation theory of a group is captured by its universal $C^{*}$-algebra and so  it is natural to  replace the algebra of continuous functions on $\TT$ with the $C^{*}$-algebra of the group. One can then characterize singularity of representations of an arbitrary group in a manner similar to the preceding paragraph. In short, we abstract the harmonic analysis for the case $\Gamma=\ZZ$ to noncommutative harmonic analysis for a nonabelian group. This approach uses essentially no structure of the group and removes the orbit equivalence techniques in the approach of \cite{Dooley}, which are only valid in the case of amenable groups.

	Although this approach using noncommutative harmonic analysis may make our methods seem abstract and esoteric, the proof of Theorem \ref{T:Koopmanthingsetc} is essentially self-contained. After the input of the aforementioned noncommutative harmonic analysis techniques, as well as basic facts about Borel measures on Polish spaces, the rest of the proof of Theorem \ref{T:Koopmanthingsetc} is elementary. It only relies on basic consequences of the finite-dimensional spectral theorem and simple volume counting estimates. Additionally, the noncommutative harmonic analysis techniques used lie at the basics of $C^{*}$-algebra theory.  On the other hand, the fact that every action is amenable groups is orbit equivalent to the integers is fairly deep. So in the amenable case we have discovered a more direct and elementary proof of Corollary \ref{C:WeakSInaiINtro}.

 	 Intuitively, entropy should be some measure of randomness of the system. Theorem \ref{T:Koopmanthingsetc} shows indeed that positive entropy actions must exhibit some randomness properties. For example, we can view the left regular representation as the representation which exhibits the perfect amount of mixing. We can also use Theorem \ref{T:Koopmanthingsetc} to  show that highly structured actions, like compact actions, must have nonpositive sofic entropy.
\begin{cor}\label{C:Kstuff} Let $\Gamma$ be a countable discrete sofic group with sofic approximation $\Sigma.$ Suppose that $\alpha\colon \Gamma\to \Aut(X,\mu)$ has image contained in a compact (for the weak topology) subgroup of $\Aut(X,\mu).$ Then for the action given by $\alpha,$
\[h_{\Sigma,\mu}(X,\Gamma)\leq 0.\]
\end{cor}
In a previous version of this article, we used our techniques to prove that distal measure-theoretic action have entropy zero. After the appearance of this version, Alpeev proved in \cite{AALP} that measure distal actions of an arbitrary group have Rohklin entropy zero. Then, Burton proved in \cite{PBur} that distal actions have naive entropy zero if the group contains a copy of $\ZZ$. Naive entropy zero implies Rohklin entropy zero which in turn implies sofic entropy zero. These proofs are both elementary, whiles ours is arguably not. We have elected to remove this section from the current version of this article, as we are actually going to prove more general results in \cite{Me9} which will give structural results of a probability measure-preserving action of a sofic group relative to its Pinsker factor. The results in \cite{AALP} imply that Rohklin entropy decreases for compact extensions (and our results in \cite{Me9} prove the same result for sofic entropy) and it appears that it is unknown whether or not naive entropy decreases under compact extensions.

	An important new tool we use to prove all of these theorems is Polish models.  If $\Gamma\actson (X,\mu)$ is a probability measure-preserving action of a countable discrete group, a \emph{topological model} for the action is an action $\Gamma \actson (Y,\nu)$  isomorphic to $\Gamma\actson (X,\mu)$ where $Y$ is a separable, metrizable topological space, $\nu$ is a $\Gamma$-invariant Borel probability measure on $Y$ and the action is by homeomorphisms. Roughly one can think of this as giving a topology on $X$ so that the action is by homeomorphisms. For actions on  standard probability spaces compact models always exist. Moreover, Kerr-Li in \cite{KLi} show that one can compute entropy in the presence of a compact model in a manner which uses the topology. Many of the computations for sofic entropy have used the compact model formalism (see \cite{BowenLi},\cite{KLi},\cite{Me5}). To prove the above theorems, we give a definition of sofic entropy in the presence of a Polish model (that is where $Y$ is merely assumed to be a Polish space, i.e. a completely metrizable separable topological space). We remark here that R. Bowen defined (see \cite{RBow}) topological entropy for uniformly continuous automorphisms of a metric space, proving that it was invariant under uniformly continuous conjugacies. Our approach is slightly different here, we do not require the homeomorphisms to be uniformly continuous and R. Bowen did not consider measure-preserving actions.

	Since compact models always exist, we should mention why we decided to consider the case of a Polish model. For this, let us mention a natural way of obtaining Polish models. Given a probability measure-preserving action $\Gamma\actson (X,\mathcal{M},\mu),$ we say that a family $\mathcal{F}$ of measurable functions $X\to\CC$ is $\emph{generating}$ if $\mathcal{M}$ is the smallest complete, $\Gamma$-invariant sigma-algebra of sets in $X$ which makes all the elements of $\mathcal{F}$ measurable. Associated to a family of generators, one can canonically produce a topological model in the following manner. Define
	\[\Phi\colon X\to \CC^{\mathcal{F}\times \Gamma}\]
by
\[\Phi(x)(f,g)=f(g^{-1}x)\]
and let $\Gamma\actson \CC^{\mathcal{F}\times \Gamma}$ be the shift action given by
\[(gx)(f,h)=x(f,g^{-1}h).\]
Setting $\nu=\Phi_{*}\mu,$ we have that $\Gamma\actson (\CC^{\mathcal{F}\times\Gamma},\nu)$ is a topological model of our action. We can thus think of topological models as a ``presentation-theoretic'' approach to ergodic theory analogous to presentation theory of groups. If $\mathcal{F}\subseteq L^{\infty}(X,\mu),$ then the topological model one produces above is a compact model. In general the topological model is Polish. Thus Polish models are a canonical way of dealing with unbounded generators. This is relevant because all of our assumptions are about functions in $L^{2}(X,\mu)$ and not $L^{\infty}(X,\mu)$. There are of course ways of turning a family of generators which are unbounded into family of generators which are bounded. One can employ cut-off functions, or compose with injective continuous maps $\RR\to [-1,1].$ We warn the reader that \emph{all} of these attempts to reduce to a family of bounded generators destroy our representation-theoretic hypotheses and so we will really need to deal with unbounded generators and, as a consequence, Polish models.

	We remark here that if we take the above topological model associated to family of generators in $L^{\infty}(X,\mu)$ then we essentially recover the operator algebraic approach to sofic entropy by Kerr-Li in \cite{KLi}. Thus the Polish model approach to sofic entropy may be regarded as a generalization of the operator algebra approach given by Kerr-Li in \cite{KLi} to a family of unbounded generators. The crucial aspect of Polish spaces which allows us to equate our definition of entropy for a Polish model to that of Bowen and Kerr-Li is tightness of a single probability measure on a Polish space. Tightness roughly asserts that, up to a small error, the probability measure the probability measure is supported on a compact set. There are examples of separable metrizable spaces where not every Borel probability measure is tight (e.g. see the remarks after Theorem 23 of \cite{Varad}) and this is the reason we need our spaces to be Polish. The assumption of a topological model being Polish is also natural since the canonical topological model associated to a family of generators is always Polish.


\textbf{Acknowledgments} Much of this work was done while I was still a PhD student at UCLA. I am very grateful to the kind hospitality and stimulating environment at UCLA. I would  like to thank Lewis Bowen for suggesting the problem of computing sofic entropy of Gaussian actions in the program ``Von Neumann algebras and ergodic theory of group actions" at Insitut Henri Poincar\'{e} in 2011. The solution to this problem ultimately led to this work. I would like to thank Stephanie Lewkiewicz for many interesting discussion about Polish models.
\section{Preliminaries}

\subsection{Notational Remarks}

	We will use $\rho$ for a representation. We will thus have to forego the usual practice in sofic entropy of using $\rho$ for a metric, and will instead use $\Delta.$

If $A,B$ are sets we use $B^{A}$ for all functions $f\colon A\to B.$ If $A=\{1,\dots,n\}$ we will use $B^{n}$ instead of  $B^{\{1,\dots,n\}}.$  If $C$ is another set and $\phi\colon B\to C$ is a function, we use $\phi^{n}\colon B^{n}\to C^{n}$ for the map
\[\phi^{n}(f)=f\circ \phi.\]
If $X$ is a Polish space, and $f\in C_{b}(X)$ we use $\|f\|$ for the uniform norm of $f.$ We will use $\|f\|_{C_{b}(X)}$ if the space $X$ is not clear from the context.

Let $(A,\Delta)$ be a pseudometric space. For subsets $C,B$ of $A,$ and $\varepsilon>0$ we say that $C$ is $\varepsilon$-contained in $B$ and write $C\subseteq_{\varepsilon}B$ if for all $c\in C,$ there is a $b\in B$ so that $\Delta(c,b)<\varepsilon.$ We say that $S\subseteq A$ is $\varepsilon$-dense if $A\subseteq_{\varepsilon}S.$ We use $S_{\varepsilon}(A,\Delta)$ for the smallest cardinality of a $\varepsilon$-dense subset of $A.$ If $C\subseteq_{\delta}B$ are subsets of $A,$ then
\[S_{2(\varepsilon+\delta)}(C,\Delta)\leq S_{\varepsilon}(B,\Delta).\]
If $A$ is a finite set we will use $u_{A}$ for the uniform probability measure on $A.$ We will typically write $u_{n}$ instead of $u_{\{1,\dots,n\}}.$ For $x\in \CC^{n},$ $\|x\|_{2}$ will denote the $\ell^{2}$ norm with respect to $u_{n}$ unless otherwise stated. There are certain times when we will have to use $\|\cdot\|_{\ell^{2}(n)},\|\cdot\|_{\ell^{2}(n,u_{n})}$ at the same time. In such an instance we will use notation which specifies which $\ell^{2}$-norm we are using. Additionally, we will use
\[\ip{\cdot,\cdot}_{\ell^{2}(n)},\]
\[\ip{\cdot,\cdot}_{\ell^{2}(n,u_{n})}\]
for the inner products on $\ell^{2}(n),\ell^{2}(n,u_{n})$ when there is potential confusion. If it is not otherwise specified, then
\[\ip{\cdot,\cdot}\]
refers to the inner product with respect to  $\ell^{2}(n,u_{n}).$

We say  that $N\subseteq A$ is $\varepsilon$-separated if for every $n_{1}\ne n_{2}$ in $N$ we have $\Delta(n_{1},n_{2})>\varepsilon.$ We use $N_{\varepsilon}(A,\Delta)$ for the smallest cardinality of a $\varepsilon$-separated subset of $A.$ Note that
\begin{equation}\label{E:separationspanning}
N_{2\varepsilon}(A,\Delta)\leq S_{\varepsilon}(A,\Delta)\leq N_{\varepsilon}(A,\Delta),
\end{equation}
and that if $A\subseteq B,$ then
\[N_{\varepsilon}(A,\Delta)\leq N_{\varepsilon}(B,\Delta).\]
\subsection{Preliminaries on Sofic Groups}\label{S:sofic}

We use $S_{n}$ for the symmetric group on $n$ letters. If $A$ is a  set, we will use $\Sym(A)$ for the set of bijections of $A.$
\begin{definition}\label{S:soficdefn}\emph{Let $\Gamma$ be a countable discrete group. A} sofic approximation of $\Gamma$ \emph{is a sequence $\Sigma=(\sigma_{i}\colon \Gamma\to S_{d_{i}})$ of functions (not assumed to be homomorphisms) so that}
\[\lim_{i\to\infty}u_{d_{i}}(\{1\leq k\leq d_{i}:\sigma_{i}(g)\sigma_{i}(h)(k)=\sigma_{i}(gh)(k)\})=1,\mbox{\emph{for all $g,h\in \Gamma$}}\]
\[\lim_{i\to\infty} u_{d_{i}}(\{1\leq k\leq d_{i}:\sigma_{i}(g)(k)\ne \sigma_{i}(h)(k)\})=1,\mbox{\emph{for all $g\ne h$ in $\Gamma.$}}\]
\emph{We will call $\Gamma$} sofic \emph{if it has a sofic approximation.}\end{definition}
It is known that all amenable groups and residually finite groups are sofic. Also, it is known that soficity is closed under free products  with amalgamation over amenable subgroups (see \cite{ESZ1},\cite{LPaun},\cite{DKP},\cite{KerrDykemaPichot2}, \cite{PoppArg}). Also graph products of sofic groups are sofic by \cite{CHR}. Additionally, residually sofic groups and locally sofic groups are sofic. Thus by Malcev's Theorem we know all linear groups are sofic. Finally, if $\Lambda$ is a subgroup of $\Gamma,$ and  $\Lambda$ is sofic and $\Gamma\actson \Gamma/\Lambda$ is amenable (in the sense of having a $\Gamma$-invariant mean) then $\Gamma$ is sofic (this can be seen by a mild generalization of the argument in Theorem 1 of \cite{ESZ1} using e.g. the observation after Definition 12.2.12 of \cite{BO}).

We will need to extend a sofic approximation to certain algebras associated to $\Gamma.$ Let $\CC(\Gamma)$ be the ring of finite formal linear combinations of elements of $\Gamma$ with addition defined naturally and multiplication defined by
\[\left(\sum_{g\in\Gamma}a_{g}g\right)\left(\sum_{h\in\Gamma}b_{h}h\right)=\sum_{g\in\Gamma}\left(\sum_{h\in\Gamma}a_{h}b_{h^{-1}g}\right)g.\]
We will also define a conjugate-linear involution on $\CC(\Gamma)$ by
\[\left(\sum_{g\in\Gamma}a_{g}g\right)^{*}=\sum_{g\in\Gamma}\overline{a_{g^{-1}}}g.\]

Given a sofic approximation $\Sigma=(\sigma_{i}\colon\Gamma\to S_{d_{i}})$ and $\alpha=\sum_{g\in\Gamma}\alpha_{g}g\in \CC(\Gamma)$ we define $\sigma_{i}(\alpha)\in M_{d_{i}}(\CC)$ by
\[\sigma_{i}(\alpha)=\sum_{g\in\Gamma}\alpha_{g}\sigma_{i}(g).\]
In order to talk about the asymptotic properties of this extended sofic approximation, we will need a more analytic object associated to $\Gamma.$

Let $\lambda\colon\Gamma\to \mathcal{U}(\ell^{2}(\Gamma))$ be  the left regular representation defined by $(\lambda(g)\xi)(h)=\xi(g^{-1}h).$ We will continue to use $\lambda$ for the linear extension to $\CC(\Gamma)\to B(\ell^{2}(\Gamma)).$ The \emph{group von Neumann algebra} of $\Gamma$ is defined by
\[\overline{\lambda(\CC(\Gamma))}^{WOT}\]
where WOT denotes the weak operator topology. We will use $L(\Gamma)$ to denote the group von Neumann algebra. Define $\tau\colon L(\Gamma)\to \CC$ by
\[\tau(x)=\ip{x\delta_{e},\delta_{e}}.\]
We leave it as an exercise to the reader to verify that $\tau$ has the following properties.

\begin{list}{ \arabic{pcounter}:~}{\usecounter{pcounter}}
\item $\tau(1)=1,$\\
\item $\tau(x^{*}x)\geq 0,$  with equality if and only if $x=0$,\\
\item $\tau(xy)=\tau(yx),$  for all $x,y\in M,$
\item $\tau$ is weak operator topology continuous.
\end{list}
We call the third property the tracial property. We will typically view $\CC(\Gamma)$ as a subset of $L(\Gamma).$ In particular, we will use $\tau$ as well for the functional on $\CC(\Gamma)$ which is just the restriction of $\tau$ on $L(\Gamma).$

	 In order to state our extension of a sofic approximation properly, we shall give a general definition. Recall that $*$-algebra is  a complex algebra equipped with an involution $*$ which is conjugate linear and antimultiplicative.

\begin{definition}\emph{ A} tracial $*$-algebra \emph{ is a pair $(A,\tau)$ where $A$ is a $*$-algebra equipped with a linear functional $\tau\colon A\to \CC$ so that}
\begin{list}{ \arabic{pcounter}:~}{\usecounter{pcounter}}
\item $\tau(1)=1,$\\
\item $\tau(x^{*}x)\geq 0,$  \emph{with equality if and only if $x=0$},\\
\item $\tau(xy)=\tau(yx),$  \emph{for all $x,y\in M,$}
\item \emph{For all $a\in A,$ there is a $C_{a}>0$ so that for all $x\in A,$ $|\tau(x^{*}a^{*}ax)|\leq C_{a}\tau(x^{*}x).$}
\end{list}
\emph{ For $a,b\in A$ we let $\ip{a,b}=\tau(b^{*}a)$ and we let $\|a\|_{2}=\tau(a^{*}a)^{1/2}.$ We let $L^{2}(A,\tau)$ be the Hilbert space completion of $A$ in this inner product. By condition $4$ of the definition, we have a representation $\lambda\colon A\to B(L^{2}(A,\tau))$ defined densely by $\lambda(a)x=ax$ for $x\in A.$ We let $\|a\|_{\infty}=\|\lambda(a)\|.$}
\end{definition}

We make $M_{n}(\CC)$ into a tracial $*$-algebra using $\tr=\frac{1}{n}\Tr$ where $\Tr$ is the usual trace. In particular, we use $\|A\|_{2}=\tr(A^{*}A)^{1/2}$ and $\|A\|_{\infty}$ will denote the operator norm.

We let $\CC[X_{1},\dots,X_{n}]$ be the free $*$-algebra on $n$-generators $X_{1},\dots,X_{n}.$  We will call elements of $\CC[X_{1},\dots,X_{n}]$ $*$-polynomials in $n$ indeterminates. For a $*$-algebra $A,$ for elements $a_{1},\dots,a_{n}\in A,$ and $P\in \CC[X_{1},\dots,X_{n}]$ we use $P(a_{1},\dots,a_{n})$ for the image of $P$ under the unique $*$-homomorphism $\CC[X_{1},\dots,X_{n}]\to A$ sending $X_{j}$ to $a_{j}.$

\begin{definition}\emph{ Let $(A,\tau)$ be a tracial $*$-algebra. An} embedding sequence \emph{is a sequence $\Sigma=(\sigma_{i}\colon A\to M_{d_{i}}(\CC))$ such that}
\[\sup_{i}\|\sigma_{i}(a)\|_{\infty}<\infty,\mbox{ \emph{for all $a\in A$}}\]
\[\|P(\sigma_{i}(a_{1}),\dots,\sigma_{i}(a_{n}))-\sigma_{i}(P(a_{1},\dots,a_{n}))\|_{2}\to 0, \mbox{\emph{ for all $a_{1},\dots,a_{n}\in A$ and all $P\in \CC[X_{1},\dots,X_{n}]$}}\]
\[\tr(\sigma_{i}(a))\to \tau(a)\mbox{\emph{ for all $a\in A$}.}\]
\end{definition}

We will frequently use the following fact: if $x_{1},\dots,x_{n}\in A$ and $P\in \CC[X_{1},\dots,X_{n}]$ then
\begin{equation}\label{E:L2convergence}
\|P(\sigma_{i}(x_{1}),\dots,\sigma_{i}(x_{n}))\|_{2}\to \|P(x_{1},\dots,x_{n})\|_{2}.
\end{equation}
To prove this, first note that as
\[\|P(\sigma_{i}(x_{1}),\dots,\sigma_{i}(x_{n}))-\sigma_{i}(P(x_{1},\dots,x_{n}))\|_{2}\to 0,\]
it suffices to handle the case $n=1$  and $P(X)=X.$ In this case,
\[\|\sigma_{i}(x)\|_{2}^{2}=\tr(\sigma_{i}(x)^{*}\sigma_{i}(x))\]
and since $\|\sigma_{i}(x)^{*}\sigma_{i}(x)-\sigma_{i}(x^{*}x)\|_{2}\to 0$ we have
\[|\tr(\sigma_{i}(x)^{*}\sigma_{i}(x))-\tr(\sigma_{i}(x^{*}x))|\to 0.\]
Since
\[\tr(\sigma_{i}(x^{*}x))\to \|x\|_{2}^{2},\]
we have proved $(\ref{E:L2convergence}).$

The proof of the next two propositions will be left to the reader.
\begin{proposition}\label{P:soficextension} Let $\Gamma$ be a countable discrete sofic group with sofic approximation $\Sigma=(\sigma_{i}\colon \Gamma\to S_{d_{i}}).$ Extend $\Sigma$ to maps $\sigma_{i}:\CC(\Gamma)\to M_{d_{i}}(\CC)$ linearly. Then $\Sigma$ is an embedding sequence of $(\CC(\Gamma),\tau).$
\end{proposition}
\begin{proposition}\label{P:perturbation} Let $(A,\tau)$ be a tracial $*$-algebra and $\Sigma=(\sigma_{i}\colon A\to M_{d_{i}}(\CC))$ be an embedding sequence. If $\Sigma'=(\sigma_{i}'\colon A\to M_{d_{i}}(\CC))$ is another sequence of functions so that
\[\sup_{i}\|\sigma_{i}'(a)\|_{\infty}<\infty,\mbox{ for all $a\in A,$}\]
\[\|\sigma_{i}(a)-\sigma_{i}'(a)\|_{2}\to 0,\mbox{ for all $a\in A,$}\]
then $\Sigma'$ is an embedding sequence.
\end{proposition}

We will in fact need to extend our sofic approximation to the group von Neumann algebra. For this, we use the following.
\begin{lemma}[Lemma 5.5 in \cite{Me}] \label{L:soficextension}Let $\Gamma$ be a countable discrete group. Then any embedding sequence for $\CC(\Gamma)$ extends to one for $L(\Gamma).$
\end{lemma}
We will  use the preceding lemma when $\Gamma$ is sofic, in combination with Proposition \ref{P:soficextension}. We will often need the following  volume-packing estimate.

\begin{lemma}\label{L:triviallemma} Let $n\in\NN,$ and $p\in B(\ell^{2}(n,u_{n}))$ an orthogonal projection. For any $M,\varepsilon>0$ we have
\[S_{\varepsilon}(Mp\Ball(\ell^{2}(n,u_{n})))\leq \left(\frac{3M+\varepsilon}{3}\right)^{2\tr(p)n}.\]
\end{lemma}
\begin{proof} Let $A\subseteq Mp\Ball(\ell^{2}(n,u_{n}))$ be a maximal $\varepsilon$-separated subset of $Mp\Ball(\ell^{2}(n,u_{n})).$ Thus $A$ is $\varepsilon$-dense. Further,
\[(M+\varepsilon/3)p\Ball(\ell^{2}(n,u_{n}))\supseteq \bigcup_{x\in A}x+(\varepsilon/3)p\Ball(\ell^{2}(n,u_{n})),\]
and the right-hand side is a disjoint union. By linear algebra, the real dimension of the image of $p$ is $2\tr(p)n.$ Thus computing volumes:
\[(M+\varepsilon/3)^{2\tr(p)n}\vol(p\Ball(\ell^{2}(n,u_{n}))\geq |A|(\varepsilon/3)^{2\tr(p)n}\vol(p\Ball(\ell^{2}(n,u_{n}))).\]
Thus
\[|A|\leq\left(\frac{3M+\varepsilon}{\varepsilon}\right)^{2\tr(p)n}.\]

\end{proof}

\section{Definition of Entropy in the Presence of a Polish Model}\label{S:Polishdefinition}

Our definition will follow the ideas in \cite{KLi2} and will use \emph{dynamically generating} pseudometrics. We need to state the definition so that it works for actions on Polish spaces. We will need to assume that the pseudometrics are bounded as this is no longer automatic in the noncompact case.

\begin{definition}\emph{Let $\Gamma$ be a countable discrete group and $X$ a Polish space with $\Gamma\actson X$ by homeomorphisms. A bounded continuous pseudometric $\Delta$ on $X$ is said to be} dynamically generating \emph{if for any $x\in X$ and any open neighborhood $U$ of $x$ in $X,$ there is a $\delta>0$ and a $F\subseteq \Gamma$ finite so that if $x\in X$ and $\max_{g\in F}d(gx,gy)<\delta,$ then $y\in U.$}\end{definition}
Recall that if $X$ is compact, and $\Gamma$ is a countable discrete group acting on $X$ by homeomorphisms and $\Delta$ is a continuous pseudometric on $X,$ then $\Delta$ is said to be dynamically generating if
\[\sup_{g\in\Gamma}\Delta(gx,gy)>0\mbox{ whenever $x\ne y$ in $X$}\]
(see e.g. \cite{Li2} Section 4). The fact that this is equivalent to our definition is an easy exercise using compactness of $X.$

When $X$ is Polish, as we shall see in the proof of Lemma \ref{L:switching} it really is necessary to require the existence of $U,\delta,F$ as in the preceding definition instead of just
\[\sup_{g\in\Gamma}\Delta(gx,gy)>0\mbox{ whenever $x\ne y$ in $X$}.\]
One way to realize why this is the correct definition is as follows: let $(X,\Delta,\Gamma)$ be as in the preceding definition and let $Y$ be $X$ modded out by the equivalence relation of $a\thicksim b$ if $\Delta(a,b)=0.$ Give $Y$ the metric
\[\overline{\Delta}([a],[b])=\Delta(a,b).\]
Consider the continuous map
\[\Phi\colon X\to Y^{\Gamma}\]
given by
\[\Phi(x)(g)=[gx],\]
where we use $[a]$ for the equivalence class of $a\in X.$ The existence of $U,\delta$ as in the definition is the \emph{precise} requirement one needs to guarantee that $\Phi$ is a homeomorphism onto its image. This will be explicitly proven in Lemma \ref{L:switching}.

If $(X,\Delta)$ is a pseudometric space, we let $\Delta_{2}$ be the pseudometric on $X^{n}$ defined by
\[\Delta_{2}(x,y)^{2}=\frac{1}{n}\sum_{j=1}^{n}\Delta(x(j),y(j))^{2}.\]

\begin{definition}\emph{ Let $\Gamma$ be a countable discrete group and $X$ a Polish space with $\Gamma\actson X$ by homeomorphisms. Let $\Delta$ be a bounded pseudometric on $X.$ For a function $\sigma\colon \Gamma\to S_{d},$ for some $d\in \NN,$ a finite $F\subseteq \Gamma,$ and a $\delta>0$ we let $\Map(\Delta,F,\delta,\sigma)$ be all functions $\phi\colon \{1,\dots,d\}\to X$ so that}
\[\max_{g\in F}\Delta_{2}(\phi\circ \sigma(g),g\phi)<\delta,\]
\end{definition}
We caution the reader that though we shall typically require $X$ to be Polish, we will not require our pseudometrics to be complete. We will typically only need to care about the topological consequences of being Polish and not any metric properties. Note that $\Map(\Delta,F,\delta,\sigma)$ does not account for the measure-theoretic structure of $X.$  Given a  Polish space $X$, a finite $L\subseteq C_{b}(X),$ a $\delta>0,$ and $\mu\in \Prob(X)$ let
\[U_{L,\delta}(\mu)=\bigcap_{f\in L}\left\{\nu\in\Prob(X):\left|\int f\,d\nu-\int f\,d\mu\right|<\delta\right\}.\]
Then $U_{L,\delta}(\mu)$ form a basis of neighborhoods of $\mu$ for the weak topology. Here $C_{b}(X)$ is  the space of bounded continuous functions on $X.$  Recall that $u_{d}$ denotes the uniform probability measure on $\{1,\dots,d\}.$
\begin{definition}\emph{Suppose that $\mu$ is a $\Gamma$-invariant Borel probability measure on $X.$ For $F\subseteq\Gamma$ finite, $\delta>0$ and $L\subseteq C_{b}(X)$ finite, and $\sigma\colon\Gamma\to S_{d}$ for some $d\in\NN$  we let $\Map_{\mu}(\Delta,F,\delta,L,\sigma)$ be the set of all $\phi\in\Map(\Delta,F,\delta,\sigma)$ so that}
\[\phi_{*}(u_{d})\in U_{L,\delta}(\mu).\]
\emph{for all $f\in L.$}\end{definition}

\begin{definition}\emph{Let $\Gamma$ be a countable discrete sofic group with sofic approximation $\Sigma=(\sigma_{i}\colon\Gamma\to S_{d_{i}}).$ Let $X$ be a Polish space with $\Gamma\actson X$ by homeomorphisms and $\mu$ a $\Gamma$-invariant, Borel, probability measure on $X.$ We define the entropy of $\Gamma\actson (X,\mu)$ by}
\[h_{\Sigma,\mu}(\Delta,\varepsilon, F,\delta,L)=\limsup_{i\to\infty}\frac{1}{d_{i}}\log N_{\varepsilon}(\Map_{\mu}(\Delta,F,\delta,L,\sigma_{i}),\Delta_{2})\]
\[h_{\Sigma,\mu}(\Delta,\varepsilon)=\inf_{\substack{F\subseteq\Gamma\textnormal{finite},\\ \delta>0,\\ L\subseteq C_{b}(X)  \textnormal{finite}}}h_{\Sigma,\mu}(\Delta,F,\delta,L,\varepsilon)\]
\[h_{\Sigma,\mu}(\Delta)=\sup_{\varepsilon>0}h_{\Sigma,\mu}(\Delta,\varepsilon).\]
\end{definition}
By (\ref{E:separationspanning}) we know that  $h_{\Sigma,\mu}(\Delta)$ is unchanged if we replace $N_{\varepsilon}$ with $S_{\varepsilon}.$ Because we use $N_{\varepsilon}$ instead of $S_{\varepsilon},$ we have $h_{\Sigma,\mu}(\Delta,\varepsilon,F,\delta,L)\leq h_{\Sigma,\mu}(\Delta,\varepsilon,F',\delta',L')$ if $F\supseteq F',\delta\leq \delta',L\supseteq L'.$

	The reader may be concerned about finiteness of the expression $N_{\varepsilon}(\Map_{\mu}(\Delta,F,\delta,L,\sigma_{i}),\Delta_{2}),$ since $X$ is not compact. We note that if $\varepsilon>0$ is given, there is a finite $L\subseteq C_{b}(X),$ and $\delta>0$ so that for any finite $F\subseteq \Gamma$ we have
\[N_{\varepsilon}(\Map_{\mu}(\Delta,F,\delta,L,\sigma_{i}),\Delta_{2})<\infty.\]
To see this, note that by Prokhorov's Theorem we may choose a $K\subseteq X$ compact so that
\[\mu(K)\geq 1-\varepsilon.\]
It is not hard to see that if $L\subseteq C_{b}(X)$ is sufficiently large and $\delta>0$ is sufficiently small, then for any $\phi\in \Map_{\mu}(\Delta,F,\delta,L,\sigma_{i}),$
\[\frac{1}{d}|\{j:\phi(j)\notin K\}|=\phi_{*}(u_{d})(K^{c})\leq 2\varepsilon.\]
Suppose that $S$ is a finite $\varepsilon$-dense subset of $K$. Let $M$ be the diameter of $(X,\Delta)$  and fix $x_{0}\in X$. If we set $A=\phi^{-1}(K),$ then we find that $u_{d_{i}}(A)\geq 1-\varepsilon.$  We can find a $\psi\in S^{A}$ so that if $\widetilde{\psi}\colon\{1,\dots,d_{i}\}\to X$ is defined by
\[\widetilde{\psi}(j)=\begin{cases}
\psi(j),& \textnormal{ if $j\in A$,}\\
x_{0},& \textnormal{ if $j\in \{1,\dots,d_{i}\}\setminus A$.}
\end{cases}\]
Then
\[\Delta_{2}(\phi,\widetilde{\psi})^{2}\leq 2M^{2}\varepsilon+\varepsilon^{2}.\]
Thus
\[N_{2(2M^{2}\varepsilon+\varepsilon^{2})^{1/2}}(\Map_{\mu}(\Delta,F,\delta,L,\sigma_{i}),\Delta_{2})\leq S_{(2M^{2}\varepsilon+\varepsilon^{2})^{1/2}}(\Map_{\mu}(\Delta,F,\delta,L,\sigma_{i}),\Delta_{2})\leq S_{\varepsilon}(K)<\infty.\]

The main goal of this section is to show that $h_{\Sigma,\mu}(\Delta)$ is the same as the measure entropy of $\Gamma\actson (X,\mu)$ as defined by Bowen and extended by Kerr-Li. Throughout we use $h_{\Sigma,\mu}(X,\Gamma)$ for sofic measure entropy as defined by \cite{Bow},\cite{KLi}. We will use the formulation of sofic entropy in terms of partitions due to Kerr in \cite{KerrPartition}.  However, we will use the terminology of observables as Bowen did in \cite{Bowenfinvariant}.

\begin{definition}\emph{Let $(X,\mathcal{M},\mu)$ be a standard probability space. Let $\mathcal{S}$ be a subalgebra of $\mathcal{M}$ (here $\mathcal{S}$ is not necessarily a $\sigma$-algebra).} A finite $\mathcal{S}$-measurable observable \emph{is a measurable map $\alpha\colon X\to A$ where $A$ is a finite set and $\alpha^{-1}(\{a\})\in \mathcal{S}$ for all $a\in A.$ If $\mathcal{S}=\mathcal{M}$ we simply call $\alpha$} a finite observable. \emph{Another finite $\mathcal{S}$-measurable observable $\beta\colon X\to B$ is said to} refine \emph{$\alpha,$ written $\alpha\leq \beta,$ if there is a  $\pi\colon B\to A$ so that $\pi(\beta(x))=\alpha(x)$ for almost every $x\in X.$ If $\Gamma$ is a countable discrete group and $\Gamma\actson(X,\mathcal{M},\mu)$ by measure-preserving transformations we say that $\mathcal{S}$ is generating if $\mathcal{M}$ is the $\sigma$-algebra generated by $\{gA:A\in\mathcal{S}\}$ (up to sets of measure zero).}
\end{definition}

For the next definition we need to set up some notation. Given a standard probability space $(X,\mathcal{M},\mu),$ a countable discrete group with $\Gamma\actson (X,\mathcal{M},\mu)$ by measure-preserving transformations, a finite observable $\alpha\colon X\to A$ and $F\subseteq\Gamma$ finite, we let
\[\widetilde{\alpha}^{F}\colon X\to A^{F}\]
be defined by
\[\widetilde{\alpha}^{F}(x)(g)=\alpha(g^{-1}x).\]

\begin{definition}\emph{Let $\Gamma$ be a countable discrete group and $\sigma\in S_{d}^{\Gamma}$ for some $d\in\NN.$ Let $(X,\mathcal{M},\mu)$ be a standard probability space and let $\mathcal{S}\subseteq\mathcal{M}$ be a subalgebra. Let $\alpha\colon X\to A$ be a finite $\mathcal{S}$ measurable-observable. Given $F\subseteq\Gamma$ finite, and $\delta>0,$ we let $\AP(\alpha,F,\delta,\sigma)$ be the set of all $\phi\colon\{1,\dots,d\}\to A^{F}$ so that}
\[\sum_{a\in A^{F}}\left|u_{d}((\phi^{-1}(\{a\}))-\mu((\widetilde{\alpha}^{F})^{-1}(\{a\}))\right|<\delta.\]
\[u_{d_{i}}(\{j:\phi(j)(g)=\phi(\sigma_{i}(g)^{-1}(j))(e)\})<\delta,\mbox{\emph{ for all $g\in F.$}}\]
\end{definition}

We now give Kerr's definition of sofic measure entropy in \cite{KerrPartition}.
\begin{definition}\emph{Let $\Gamma$ be a countable discrete sofic group with sofic approximation $\Sigma=(\sigma_{i}\colon\Gamma\to S_{d_{i}}).$ Let $(X,\mathcal{M},\mu)$ be a standard probability space and $\Gamma\actson (X,\mathcal{M},\mu)$ by measure-preserving transformations. Let $\mathcal{S}$ be a subalgebra of $\mathcal{M}.$ Let $\alpha\colon X\to A$ be a finite $\mathcal{S}$-measurable observable, and let $\beta\colon X\to B$ refine $\alpha,$ and $\pi\colon B\to A$ as in the definition of $\alpha\leq\beta.$ For any $F\subseteq \Gamma$ finite, we use}
\[\pi^{*}\colon B^{F}\to A\]
\emph{for}
\[\pi^{*}(b)=\pi(b(e)).\]
{We set}
\[h_{\Sigma,\mu}(\alpha;\beta,F,\delta)=\limsup_{i\to\infty}\frac{1}{d_{i}}\log\left|(\pi^{*})^{d_{i}}(\AP(\beta,F,\delta,\sigma_{i}))\right|\]
\[h_{\Sigma,\mu}(\alpha;\beta)=\inf_{\substack{ F\subseteq\Gamma \textnormal{ finite},\\ \delta>0}}h_{\Sigma,\mu}(\alpha;\beta,F,\delta).\]
\emph{We then set}
\[h_{\Sigma,\mu}(\alpha;\mathcal{S})=\inf_{\alpha\leq \beta}h_{\Sigma,\mu}(\alpha;\beta)\]
\[h_{\Sigma,\mu}(\mathcal{S})=\sup_{\alpha}h_{\Sigma,\mu}(\alpha;\beta)\]
\emph{where the last infimum and supremum are over all $\mathcal{S}$-measurable observables. }
\end{definition}

We need the following result of Kerr.
\begin{theorem}\label{T:KerrTheorem} Let $\Gamma$ be a countable discrete sofic group with sofic approximation $\Sigma.$ Let $(X,\mathcal{M},\mu)$ be a standard probability space with $\Gamma\actson (X,\mathcal{M},\mu)$ by measure-preserving transformations. Let $\mathcal{S}\subseteq\mathcal{M}$ be a generating subalgebra. Then
\[h_{\Sigma,\mu}(\mathcal{S})=h_{\Sigma,\mu}(X,\Gamma).\]
\end{theorem}
Additionally, one can show that
\[h_{\Sigma,\mu}(\alpha;\mathcal{S})\]
is independent of $\mathcal{S}$ if $\mathcal{S}$ generates $\mathcal{M}$ (up to sets of measure zero). In this case, we set
\[h_{\Sigma,\mu}(\alpha;\mathcal{S})=h_{\Sigma,\mu}(\alpha;X,\Gamma).\]

We now proceed to prove that our definition of sofic entropy with respect to a Polish model recovers  measure-theoretic entropy with respect to a sofic approximation. Let us briefly outline the proof. First we show that for any dynamically generating pseudometric $\Delta$ on $X,$  there is a compatible metric $\Delta'$ so that
\[h_{\Sigma,\mu}(\Delta)=h_{\Sigma,\mu}(\Delta')\]
(see Lemma \ref{L:switching}). Thus we may assume that our metrics our compatible. We then use Kerr's version of measure entropy, using the subalgebra of sets which from the point of view of the measure ``appear" to be open and closed (in a sense to be made precise later). We then show that both the topological version and the observable version of microstates produce roughly the same space (see Lemma \ref{L:COstuffman}).  The essential fact for proving the last step will be tightness of a single probability measure on a Polish space. The theorem will follow without too much difficulty from these preliminary lemmas. The following proof is a minor modification of the argument in Lemma 4.4 of \cite{Li}, as well as Lemma 6.12 of \cite{BowenGroupoid}. We have decided to include the proof to alleviate any concerns that may arise from working in the noncompact case, as well as address the necessary modifications that occur in the definition of a dynamically generating pseudometric in the Polish case.
\begin{lemma}\label{L:switching} Let $\Gamma$ be a countable discrete sofic group with sofic approximation $\Sigma=(\sigma_{i}\colon \Gamma\to S_{d_{i}}).$ Let $X$ be a Polish space with $\Gamma\actson X$ by homeomorphisms and $\mu$ a $\Gamma$-invariant probability measure. Given a dynamically generating pseudometric $\Delta$ on $X,$ there is a bounded compatible metric $\Delta'$ on $X$ so that
\[h_{\Sigma,\mu}(\Delta)=h_{\Sigma,\mu}(\Delta').\]
\end{lemma}
\begin{proof} Let $M$ be the diameter of $(X,\Delta).$ Since $\Gamma$ is countable, we may find positive real numbers $\{\alpha_{g}:g\in\Gamma\}$  with $\alpha_{e}\geq 1/2$ and
\[\sum_{g}\alpha_{g}=1.\]
Set
\[\Delta'(x,y)=\sum_{g\in\Gamma}\alpha_{g}\Delta(gx,gy).\]
We prove the lemma with this $\Delta'.$ The lemma is proved in several steps.

\emph{Step 1: We show that $\Delta'$ is a compatible metric.} For this let $Y$ be $X$ modded out by the equivalence relation $a\thicksim b$ if $\Delta(a,b)=0.$ For $a\in X,$ let $[a]$ be the equivalence class of $a.$ Make $Y$ a metric space with metric $\overline{\Delta}$ given by
\[\overline{\Delta}([a],[b])=\Delta(a,b).\]
Observe that $\overline{\Delta}$ is well-defined because $\Delta$ satisfies  the triangle inequality. Then $Y^{\Gamma}$  has a compatible metric given by
\[\Delta_{\Gamma}(a,b)=\sum_{g\in\Gamma}\alpha_{g}\overline{\Delta}(a(g),b(g)).\]
Moreover we have an injective map $\Phi\colon X\to Y^{\Gamma}$ by
\[\Phi(x)(g)=[gx],\]
as
\[\Delta_{\Gamma}(\Phi(x),\Phi(y))=\Delta'(x,y),\]
it is enough to show that $\Phi$ is a homeomorphism onto its image. It is clear that $\Phi$ is continuous. Suppose that $x\in X$ and $U$ is a neighborhood of $x\in X.$ By the definition of dynamically generating, we may choose a finite $F\subseteq\Gamma$ and a $\delta>0$ so that if $y\in X$ and
\[\max_{g\in F}\Delta(gx,gy)<\delta,\]
then $y\in U.$  If we let
\[V=\{y\in Y^{\Gamma}:\overline{\Delta}([gx],y(g))<\delta\mbox{ for all $g\in F$}\},\]
then $V$ is a neighborhood of $\Phi(x)$ in $Y^{\Gamma}$ and if $y\in X$ has $\Phi(y)\in V,$ then $y\in U.$ Thus $\Phi$ is a homeomorphism onto its image.

\emph{Step 2: We show that $h_{\Sigma,\mu}(\Delta')\leq h_{\Sigma,\mu}(\Delta).$}
Let $\varepsilon>0$ and choose a finite $E\subseteq\Gamma$ sufficiently large so that
\[\sum_{g\in\Gamma\setminus E}\alpha_{g}<\varepsilon.\]
 Let $F\subseteq\Gamma$ finite, $\delta>0$ and $L\subseteq C_{b}(X)$ finite be given. We will assume that $F\supseteq E.$
 We will choose $\delta$ sufficiently small in a manner depending upon $\varepsilon$ to be determined later. Since $\alpha_{e}\geq 1/2$ we have
\[\Delta(x,y)\leq 2\Delta'(x,y).\]
So by Minkowski's inequality
\[\Map_{\mu}(\Delta',F,\delta,L,\sigma_{i})\subseteq \Map_{\mu}(\Delta,F,4\delta,L,\sigma_{i}).\]
For $\phi,\psi\in\Map_{\mu}(\Delta,F,2\delta,L,\sigma_{i})$ we have by Minkowski's inequality,
\begin{align*}
\Delta'_{2}(\phi,\psi)&\leq \sum_{g\in\Gamma}\alpha_{g}\Delta_{2}(g\phi,g\psi)\\
&\leq \varepsilon M+\sum_{g\in E}\alpha_{g}\Delta_{2}(g\phi,g\psi)\\
&\leq \varepsilon M+2\delta+\sum_{g\in E}\alpha_{g}\Delta_{2}(\phi\circ \sigma_{i}(g),\Delta\circ \sigma_{i}(g))\\
&\leq \varepsilon M+2\delta+\Delta_{2}(\phi,\psi)
\end{align*}
where in the last line we use that $\Delta_{2}(\phi\circ \sigma_{i}(g),\psi\circ\sigma_{i}(g))=\Delta_{2}(\phi,\psi).$  Thus for all sufficiently small $\delta$
\[S_{3\varepsilon(M+2)}(\Map_{\mu}(\Delta',F,\delta,L,\sigma_{i}),\Delta_{2}')\leq S_{\varepsilon}(\Map_{\mu}(\Delta,F,4\delta,L,\sigma_{i})).\]
Thus by (\ref{E:separationspanning})  we have
\[h_{\Sigma,\mu}(\Delta',6\varepsilon(M+2))\leq h_{\Sigma,\mu}(\Delta,\varepsilon,F,4\delta,L).\]
As $h_{\Sigma,\mu}(\Delta,\varepsilon,F,4\delta,L)$ is monotone in $(F,4\delta,L)$ we let $F,L$ increase to $\Gamma,C_{b}(X)$ and take $\delta\to 0$ to find that
\[h_{\Sigma,\mu}(\Delta',6\varepsilon(M+2))\leq h_{\Sigma,\mu}(\Delta',\varepsilon)\leq h_{\Sigma,\mu}(\Delta).\]
Now letting $\varepsilon\to 0$ completes the proof of Step 2.

\emph{Step 3: We show that $h_{\Sigma,\mu}(\Delta)\leq h_{\Sigma,\mu}(\Delta')$.} Let $\varepsilon>0.$ Suppose we are given finite $F'\subseteq \Gamma,L'\subseteq C_{b}(X)$ and $\delta'>0.$ Set $L=L'$ and let $F\subseteq \Gamma$ be a sufficiently large finite set depending upon $F',\delta'$ in a manner to determined later and set $\delta=\delta'$. Choose a finite $E\subseteq \Gamma$ so that
\[\sum_{g\in\Gamma\setminus E}\alpha_{g}<\delta'.\]
If $\phi\in\Map_{\mu}(\Gamma,F,\delta,L,\sigma_{i})$ and $h\in F',$  we have by Minkowski's inequality
\begin{align*}
\Delta_{2}'(h\phi,\phi\circ\sigma_{i}(h))&\leq\sum_{g\in\Gamma}\alpha_{g}\Delta_{2}(gh\phi,g\phi\circ\sigma_{i}(h))\\
&\leq \delta'+\sum_{g\in F}\alpha_{g}\Delta_{2}(gh\phi,\phi\circ\sigma_{i}(gh))+\sum_{g\in F}\alpha_{g}\Delta_{2}(\phi\circ\sigma_{i}(gh),\phi\circ\sigma_{i}(g)\sigma_{i}(h))\\
&+\sum_{g\in F}\alpha_{g}\Delta_{2}(\phi\circ\sigma_{i}(g)\sigma_{i}(h),g\phi\circ\sigma_{i}(h))\\
&=\delta'+\sum_{g\in F}\alpha_{g}\Delta_{2}(gh\phi,\phi\circ\sigma_{i}(gh))+\sum_{g\in F}\alpha_{g}\Delta_{2}(\phi\circ\sigma_{i}(gh),\phi\circ\sigma_{i}(g)\sigma_{i}(h))\\
&+\sum_{g\in F}\alpha_{g}\Delta_{2}(\phi\circ \sigma_{i}(g),g\phi)\\
&<\delta'+2\delta+\sum_{g\in E}\alpha_{g}\Delta_{2}(\phi\circ \sigma_{i}(g)\sigma_{i}(h),\phi\circ\sigma_{i}(gh)),
\end{align*}
if we force $F\supseteq EF'\cup E.$  Set $\delta=\delta',$ as
\[\Delta_{2}(\phi\circ \sigma_{i}(g)\sigma_{i}(h),\phi\circ\sigma_{i}(gh))^{2}\leq Mu_{d_{i}}(\{1\leq j\leq d_{i}:\sigma_{i}(g)\sigma_{i}(h)(j)\ne \sigma_{i}(gh)(j)\})\to 0\]
for $g,h\in\Gamma$ we find that
\[\Map_{\mu}(\Delta,F,\delta,L,\sigma_{i})\subseteq\Map_{\mu}(\Delta',F',4\delta',\sigma_{i}),\]
for all large $i.$ As
\[\Delta_{2}(\phi,\psi)\leq 2\Delta'_{2}(\phi,\psi),\]
we find that
\[h_{\Sigma,\mu}(\Delta,6\varepsilon)\leq h_{\Sigma,\mu}(\Delta,6\varepsilon,F,\delta,L)\leq h_{\Sigma,\mu}(\Delta',\varepsilon,F',\delta',L').\]
Taking the infimum over $F',\delta',L'$ we find that
\[h_{\Sigma,\mu}(\Delta,6\varepsilon)\leq h_{\Sigma,\mu}(\Delta',\varepsilon)\leq h_{\Sigma,\mu}(\Delta').\]
Letting $\varepsilon\to 0$ proves Step 3.

\end{proof}

To prove the theorem, we need to single out a nice subalgebra of measurable sets. Let $X$ be a Polish space and $\mu$ a Borel probability measure on $X.$ We let $\CO_{\mu}$ be the set of Borel sets $E$ so that
\[\mu(\Int E)=\mu(\overline{E}).\]
Note that $\CO_{\mu}$ is an algebra of sets. These sets are often called continuity sets in the literature.

For the next lemma we need some notation. Given a metric space $(X,\Delta),$ $E\subseteq X$ and $\varepsilon>0$ we let
\[E_{\varepsilon}=\{x\in E:\Delta(x,E^{c})\geq \varepsilon\}\]
\[\mathcal{O}_{\varepsilon}(E)=\{x\in X:\Delta(x,E)<\varepsilon\}.\]
\begin{lemma}\label{L:CO} Let $X$ be a Polish space and $\mu$ a Borel probability measure on $X.$ Let $\Delta$ be a compatible metric on $X.$  Given $E\in \CO_{\mu},$  and $\eta>0$ there is a neighborhood $U$ of $\mu$ in the weak topology and a $\kappa>0$ so that if $\nu\in U\cap \Prob(X)$ then
\[|\nu(E)-\mu(E)|<\eta\]
\[\nu(\mathcal{O}_{\kappa}(E)\setminus E_{\kappa})<\eta.\]
\end{lemma}

\begin{proof}

It is a consequence of the Portmanteau Theorem that a sequence of probability measure $\mu_{n}$ on $X$ converges weakly to $\mu$ if and only if $\mu_{n}(A)\to \mu(A)$ for every continuity set $A$ of $\mu.$ Thus we can choose a neighborhood $U_{1}$ of $\mu$ so that
\[|\nu(E)-\mu(E)|<\eta\]
for all $\nu\in U_{1}.$  To obtain the second estimate, again by the Portmanteau Theorem we can choose a neighborhood $U_{2}$ of $\mu$ so that
\[\nu(\overline{\mathcal{O}_{2\kappa}(E)\setminus E_{2\kappa}})\leq \mu(\overline{\mathcal{O}_{2\kappa}(E)\setminus E_{2\kappa}})+\frac{\eta}{2}\]
for all $\nu\in U_{2}.$  We have $\mathcal{O}_{\kappa}(E)\setminus E_{\kappa}\subseteq \overline{\mathcal{O}_{2\kappa}(E)\setminus E_{2\kappa}},$ so
\[\nu(\mathcal{O}_{\kappa}(E)\setminus E_{\kappa})\leq\mu(\overline{\mathcal{O}_{2\kappa}(E)\setminus E_{2\kappa}})+\frac{\eta}{2}.\]
Since $E$ is a continuity set of $\mu$ we can choose $\kappa$ small enough so that
\[\mu(\overline{\mathcal{O}_{2\kappa}(E)\setminus E_{2\kappa}})<\frac{\eta}{2}.\]
The lemma is now completed by setting $U=U_{1}\cap U_{2}.$

\end{proof}

Given $\sigma\colon\Gamma\to S_{d}$ for some $d\in\NN,$ and $\phi\in A^{d}$ we shall define $\phi^{F}\colon\{1,\dots,d\}\to A^{F}$ by'
\[\phi_{\sigma}^{F}(j)(g)=\phi(\sigma(g)^{-1}(j)), \mbox{ for $g\in F$}\]

\begin{lemma}\label{L:COstuffman} Let $\Gamma$ be a countable discrete sofic group with sofic approximation $\Sigma=(\sigma_{i}\colon\Gamma\to S_{d_{i}}).$ Let $X$ be a Polish space with $\Gamma\actson X$ by homeomorphisms and $\mu$ a $\Gamma$-invariant Borel probability measure on $X.$ Let $\Delta$ be a bounded compatible metric on $X.$

(i): Let $\beta\colon X\to B$ a finite $\CO_{\mu}$-measurable observable. Given $F\subseteq\Gamma$ finite, $\delta>0$ there are finite $F'\subseteq\Gamma,L'\subseteq C_{b}(X),$ and $\delta'>0$ so that if $\phi\in\Map_{\mu}(\Delta,F',\delta',L',\sigma_{i})$ then $(\beta\circ\phi)_{\sigma_{i}}^{F}\in \AP(\beta,F,\delta,\sigma_{i}).$

(ii): Given finite $F'\subseteq\Gamma,L'\subseteq C_{b}(X)$ and $\delta'>0,$ there are finite $F\subseteq\Gamma$ and a $\delta>0$ and a finite $\CO_{\mu}$-measurable observable $\beta\colon X\to B$ so that if $\widetilde{\beta}^{F}\circ\phi\in\AP(\beta,F,\delta,\sigma_{i})$ then $\phi\in\Map_{\mu}(\Delta,F',\delta',L',\sigma_{i}).$

\end{lemma}

\begin{proof}

(i): Let $\eta>0$ be sufficiently small depending upon $\delta$ in a manner to be determined later. By the preceding lemma, we may find a $L'\subseteq C_{b}(X)$ finite, and a $\delta'>0$ so that if $\nu\in\Prob(X)$ and
\[\left|\int_{X}f\,d\mu-\int_{X}f\,d\nu\right|<\delta'\mbox{ for all $f\in L'$}\]
then for all $E\subseteq F$  and for all $b\in B^{E},$
\[|\mu((\widetilde{\beta}^{E})^{-1}(\{b\}))-\nu((\widetilde{\beta}^{E})^{-1}(\{b\}))|<\eta,\]
\[\nu(\mathcal{O}_{\sqrt{\delta'}}(\widetilde{\beta}^{E})^{-1}(\{b\}))\setminus (\widetilde{\beta}^{E})^{-1}(\{b\})_{\sqrt{\delta'}})\leq \eta.\]
Set $F'=F.$ Suppose that $\phi\in\Map_{\mu}(\Delta,F',\delta',L',\sigma_{i}).$ Note that if we choose $\eta$ sufficiently small, then we have forced
\[\sum_{b\in B^{F}}|\mu((\widetilde{\beta}^{F})^{-1}(\{b\}))-u_{d_{i}}(((\beta\circ\phi)_{\sigma_{i}}^{F})^{-1}(\{b\}))|<\delta\]
for all $\phi\in \Map(\Delta,F',\delta',L',\sigma_{i}).$  Since $(\beta\circ \phi)^{F}_{\sigma_{i}}(j)(g)=(\beta\circ \phi)^{F}(\sigma_{i}(g)^{-1}(j))(e),$ this means $(\beta\circ \phi)_{\sigma_{i}}^{F}\in\AP(\beta,F,\delta,\sigma_{i}).$

(ii): Let $M>0$ be the diameter of $(X,\Delta).$ Let $\kappa>0$ be sufficiently small depending upon $F',\delta'$ in a manner to be determined later.  Since $X$ is Polish,  Prokhorov's Theorem applied to $\{\mu\}$ implies that we can find a compact set $K\subseteq X$ so that
\[\mu(X\setminus K)\leq \kappa.\]
 Since $K$ is compact, we can find points $x_{1},\dots,x_{n}\in K,$ and numbers $\delta_{j}\in (0,\kappa),j=1,\dots,n$ so that if $B(x,\alpha)$ is the ball in $X$ of radius $\alpha$ with respect to $\Delta$ then
\[\mu(\overline{B(x_{j},\delta_{j})}\setminus B(x_{j},\delta_{j}))=0,\]
\[\sup_{x,y\in B(x_{j},\delta_{j})}|f(x)-f(y)|<\delta'\mbox{ for all $f\in L'$},\]
\[K\subseteq\bigcup_{j=1}^{n}B(x_{j},\delta_{j}).\]
Let
\[E=X\setminus \bigcup_{j=1}^{n}B(x_{j},\delta_{j})\]
and define $\beta\colon X\to \{0,1\}^{n+1}$ by
\[\beta(x)(k)=\begin{cases}
\chi_{B(x_{k},\delta_{k})}(x),&\textnormal{ if $1\leq k\leq n$}\\
\chi_{E}(x),&\textnormal{ if $k=n+1$.}
\end{cases}\]
Note that $\beta$ is $\CO_{\mu}$- measurable. Set $F=F'\cup\{e\}\cup (F')^{-1},$ and let $\delta>0$ be sufficiently small in a manner to depend upon $L',\delta',F'$ to be determined later. We will assume that
\[\delta<\min\{\delta_{j}:j=1,\dots,n\}.\]
 Suppose that $\widetilde{\beta}^{F}\circ\phi\in\AP(\beta,F,\delta,\sigma_{i}).$ Given $g\in F,$ let
\[C_{g}=\{j:\beta(g\phi(j))=\beta(\phi(\sigma_{i}(g)(j)))\}.\]
Since
\[\beta(g\phi(j))=(\widetilde{\beta}^{F}\circ \phi)(j)(g^{-1}),\]
\[\beta(\phi(\sigma_{i}(g)(j)))=(\widetilde{\beta}^{F}\circ \phi)(\sigma_{i}(g))(j)(e),\]
and
\[u_{d_{i}}(\{j:\sigma_{i}(g^{-1})^{-1}(j)=\sigma_{i}(g)(j)\})\to_{i\to\infty}1\]
we see that for all large $i,$
\[u_{d_{i}}(C_{g})\geq 1-2\delta.\]
For $j\in C_{g},$ we necessarily have
\[\Delta(\phi(\sigma_{i}(g)(j)),g\phi(j))<\kappa.\]
Thus
\[\Delta_{2}(\phi\circ \sigma_{i}(g),g\phi)^{2}\leq \kappa^{2}+M^{2}\delta.\]
So if we choose $\kappa<\delta,$ and then $\delta$ sufficiently small, we have forced
\[\Delta_{2}(\phi\circ \sigma_{i}(g),g\phi)<2\delta'.\]
 We will want to force $\kappa,\delta'$ to be even smaller later. Using $\|f\|$ for the uniform norm of $f\in C_{b}(X),$ we have for all $f\in L',$
\begin{align*}
\left|\int_{X}f\,d\mu-\int_{X}f\,d\phi_{*}(u_{d_{i}})\right|&\leq \kappa\|f\|+u_{d_{i}}(\phi^{-1}(E))\|f\|+\sum_{j=1}^{n}\left|\int_{B(x_{j},\delta_{j})}f\,d\mu-\int_{B(x_{j},\delta_{j})}f\,d\phi_{*}(u_{d_{i}})\right|\\
&\leq \kappa\|f\|+2\kappa\|f\|+\delta'+\sum_{j=1}^{n}\left|f(x_{j})\mu(B(x_{j},\delta_{j}))-\frac{1}{d_{i}}\sum_{k:\phi(k)\in B(x_{j},\delta_{j})}f(\phi(k))\right|\\
&\leq \kappa\|f\|+2\kappa\|f\|+\delta'+\delta'\sum_{j=1}^{n}\phi_{*}(u_{d_{i}})(B(x_{j},\delta_{j}))\\
&+\sum_{j=1}^{n}\left|f(x_{j})\mu(B(x_{j},\delta_{j}))-\frac{1}{d_{i}}\sum_{k:\phi(k)\in B(x_{j},\delta_{j})}f(x_{j})\right|\\
&\leq \kappa\|f\|+2\kappa\|f\|+2\delta'+\sum_{j=1}^{n}|f(x_{j})| |\mu(B(x_{j},\delta_{j})-\phi_{*}(u_{d_{i}})(B(x_{j},\delta_{j}))|\\
&\leq \kappa\|f\|+2\kappa\|f\|+2\delta'+n\|f\|\delta.
\end{align*}
We may choose $\kappa<\frac{\delta'}{\|f\|}.$ This forces $n$ on us, but we may then choose $\delta$ sufficiently small so that
$\phi\in\Map_{\mu}(\Delta,F',6\delta',L',\sigma_{i}).$ As $\delta'$ is arbitrary this completes the proof.

\end{proof}

We are now ready to show that our definition of entropy in the case of a Polish model agrees with the usual measure entropy.

\begin{theorem}\label{T:PolishModels} Let $\Gamma$ be a countable discrete sofic group with sofic approximation $\Sigma.$ Let $X$ be a Polish space with $\Gamma\actson X$  by homeomorphisms and $\mu$ a $\Gamma$-invariant, Borel, probability measure on $X.$ For any dynamically generating pseudometric $\Delta$ on $X$  we have
\[h_{\Sigma,\mu}(\Delta)=h_{\Sigma,\mu}(X,\Gamma).\]
\end{theorem}

\begin{proof}

Let $\Sigma=(\sigma_{i}\colon\Gamma\to S_{d_{i}}).$ By Lemma \ref{L:switching} we may assume that $\Delta$ is a bounded compatible metric on $X.$ Let $M$ be the diameter of $(X,\rho).$ We will apply Theorem \ref{T:KerrTheorem} with $\mathcal{S}=\CO_{\mu}.$ We leave it as an exercise to show that for all $x\in X$ we have
\[\mu(B(x,r))=\mu(\overline{B(x,r)})\]
for all but countably many $r\in (0,\infty).$ Thus $\CO_{\mu}$ generates the $\sigma$-algebra of Borel subsets of $X.$ We first show that
\[h_{\Sigma,\mu}(\Delta)\leq h_{\Sigma,\mu}(\CO_{\mu}).\]
Let $\varepsilon>0.$ Since $X$ is Polish, we may apply Prokhorov's Theorem to find a compact $K\subseteq X$ so that
\[\mu(X\setminus K)<\varepsilon.\]
By compactness of $K,$ we find $x_{1},\dots,x_{n}\in K,$ and $\varepsilon>\delta_{1},\dots,\delta_{n}>0$ so that
\[K\subseteq \bigcup_{j=1}^{n}B(x_{j},\delta_{j}),\]
\[B(x_{j},\delta_{j})\in\CO_{\mu}.\]
Set
\[E=X\setminus \bigcup_{j=1}^{n}B(x_{j},\delta_{j}).\]
Define
\[\alpha\colon X\to \{0,1\}^{n+1}\]
by
\[\alpha(x)(k)=\begin{cases}
\chi_{B(x_{k},\delta_{k})}(x),& \textnormal{ if $1\leq k\leq n$}\\
\chi_{E}(x),& \textnormal{ if $k=n+1$.}
\end{cases}\]
Let $\beta\colon X\to B$ be any finite $\CO_{\mu}$-measurable observable refining $\alpha$ and let $\pi\colon B\to A$ be as in the definition of $\beta\geq\alpha.$ Suppose we are given a finite $F\subseteq\Gamma$ and a $\delta>0.$ By the preceding Lemma, we may find a finite $F'\subseteq\Gamma,L'\subseteq C_{b}(X),$ and a $\delta'>0$ so that
\[\beta^{F}_{\sigma_{i}}(\Map_{\mu}(\Delta,F',\delta',L',\sigma_{i}))\subseteq \AP(\beta,F,\delta,\sigma_{i}).\]
By Lemma \ref{L:CO}, we may assume that $L'$ is sufficiently large so that
\[\phi_{*}(u_{d_{i}})(X\setminus E)\leq 2\varepsilon\]
for all $\phi\in\Map_{\mu}(\Delta,F',\delta',L',\sigma_{i}).$ Choose elements $\{\phi_{s}\}_{s\in S}$ with $\phi_{s}\in \Map_{\mu}(\Delta',F',L',\delta',\sigma_{i})$ where $S$ is some index set so that
\[\{\alpha\circ\phi_{s}:s\in S\}=\alpha^{d_{i}}(\Map_{\mu}(\Delta,F',\delta',L',\sigma_{i}))\]
and so that
\[\alpha \circ\phi_{s}\ne \alpha \circ\phi_{s'}\mbox{ for $s\ne s'$ in $S$}.\]
Then
\[|S|\leq |(\pi^{*})^{d_{i}}(\AP(\beta,F,\delta,\sigma_{i}))|.\]
Let $\phi\in\Map_{\mu}(\Delta,F',\delta',L',\sigma_{i})$ and let $s\in S$ be such that $\alpha\circ \phi=\alpha\circ\phi_{s}.$ Then
\[\Delta_{2}(\phi,\phi_{s})^{2}\leq 4M^{2}\varepsilon+\frac{1}{d_{i}}\sum_{j:\phi(j),\phi_{s}(j)\notin E}\Delta(\phi(j),\phi_{s}(j))^{2}.\]
If $\phi(j)$ and $\phi_{s}(j)$ are not in $E,$ then the fact that $\alpha(\phi(j))=\alpha(\phi_{s}(j))$ implies  $\Delta(\phi(j),\phi_{s}(j))<\varepsilon,$ so
\[\Delta_{2}(\phi,\phi_{s})^{2}<4M^{2}\varepsilon+\varepsilon^{2}.\]
Thus by (\ref{E:separationspanning}),
\[h_{\Sigma,\mu}(\Delta,2(4M\varepsilon+\varepsilon^{2})^{1/2})\leq h_{\Sigma,\mu}(\alpha;\beta,F,\delta).\]
Taking the infimum over all $\beta,F,\delta$ we find that
\[h_{\Sigma,\mu}(\Delta,2(4M\varepsilon+\varepsilon^{2})^{1/2})\leq h_{\Sigma,\mu}(\alpha)\leq h_{\Sigma,\mu}(X,\Gamma),\]
and letting $\varespilon\to 0$ implies that
\[h_{\Sigma,\mu}(\Delta)\leq h_{\Sigma,\mu}(X,\Gamma).\]

For the reverse inequality, let $\alpha\colon X\to A$ be a $\CO_{\mu}$-measurable finite observable. Fix $\kappa>0,$ and let $\kappa'>0$ depend upon $\kappa$ in a manner to be determined later.  By Lemma $\ref{L:CO}$ we may choose $\eta>0$ and $L_{0}\subseteq C_{b}(X)$ finite so that if $\nu\in \Prob(X)$ and
\[\left|\int_{X}f\,d\mu-\int_{X}f\,d\nu\right|<\eta,\]
 for all $f\in L_{0}$ then
\[\left|\nu(\alpha^{-1}(\{a\}))-\mu(\alpha^{-1}(\{a\}))\right|<\kappa',\]
\[\nu(\mathcal{O}_{\eta}(\alpha^{-1}(\{a\})\setminus \alpha^{-1}(\{a\})_{\eta})<\kappa'.\]
Let $F'\subseteq \Gamma,L'\subseteq C_{b}(X)$ be given finite sets and $\delta'>0$ be given. We may assume that $L'\supseteq L_{0}.$ By the preceding Lemma, we may choose a refinement $\beta\colon X\to B$ of $\alpha,$  a finite $F\subseteq\Gamma,$ and a $\delta>0$ so that if $\phi\in X^{d_{i}}$ and $\widetilde{\beta}^{F}\circ \phi\in \AP(\beta,F,\delta,\sigma_{i}),$ then $\phi\in\Map_{\mu}(\Delta,F',\delta',L',\sigma_{i}).$ Choose $\pi\colon B\to A$ so that $\alpha=\pi\circ \beta$ and choose a map $s\colon B^{F}\to X$ so that $\id=\widetilde{\beta}^{F}\circ s.$ By construction if $\phi\in\AP(\beta,F',\delta',\sigma_{i})$ we have
\[s\circ \phi\in\Map_{\mu}(\Delta,F',\delta',L',\sigma_{i}).\]

	Let $\varepsilon>0$ be sufficiently small depending upon $\eta$ to be determined later. Let $\{\phi_{t}:t\in T\}$ be such that
\[\phi_{t}\in \AP(\beta,F,\delta,\sigma_{i})\mbox{ for all $t\in T$,}\]
\[\{s\circ \phi_{t}:t\in T\}\mbox{ is  $\varepsilon$-dense in} \{s\circ \phi:\phi\in\AP(\beta,F,\delta,\sigma_{i})\},\]
\[\phi_{t}\ne\phi_{t'}\mbox{ for $t\ne t'$}.\]
We may (and will) choose $T$ so that
\[|T|\leq S_{\varepsilon/2}(\Map_{\mu}(\Delta,F',\delta',L',\sigma_{i}))\leq N_{\varepsilon/2}(\Map_{\mu}(\Delta,F',\delta',L',\sigma_{i})).\]
Note that
\[(\pi^{*})^{d_{i}}(\AP(\beta,F,\delta,\sigma_{i}))\subseteq \bigcup_{t\in T}\alpha^{d_{i}}(B_{\Delta_{2}}(s\circ \phi_{t},\varepsilon)\cap \Map_{\mu}(\Delta,F',\delta',L',\sigma_{i})).\]

	We thus have to bound $|\alpha^{d_{i}}(B_{\Delta_{2}}(\phi_{t},\varepsilon)\cap \Map_{\mu}(\Delta,F',\delta',L',\sigma_{i}))|$ from above. Fix $t\in T,$ Suppose that $\phi\in\Map_{\mu}(\Delta,F',\delta',L',\sigma_{i})$  and that
\[\Delta_{2}(\phi,s\circ \phi_{t})<\varepsilon.\]
Let
\[C=\bigcup_{a\in A}\{1\leq j\leq d_{i}:\phi(j)\in \mathcal{O}_{\eta}(\alpha^{-1}(\{a\}))\setminus \alpha^{-1}(\{a\})_{\eta}\}\cup\{1\leq j\leq d_{i}:s\circ \phi_{t}(j)\in \mathcal{O}_{\eta}(\alpha^{-1}(\{a\}))\setminus\alpha^{-1}(\{a\})_{\eta}\}.\]
By choosing $\kappa'$ sufficiently small, we may assume that
\[u_{d_{i}}(C)\leq \kappa.\]
Let
\[D=\{1\leq j\leq d_{i}:\Delta(\phi(j),s\circ \phi_{t}(j))\geq \sqrt{\varepsilon}\}.\]
Thus
\[u_{d_{i}}(D)\leq \sqrt{\vaerpsilon}.\]
For $j\in\{1,\dots,d_{i}\}\setminus (C\cup D),$ and $a=\alpha(\phi(j)),$ we have that $s\circ \phi_{t}(j)\in \mathcal{O}_{\sqrt{\varepsilon}}(\alpha^{-1}(\{a\})).$ If we choose $\sqrt{\varepsilon}<\eta,$ then $\alpha(s\circ \phi_{t}(j))=a$ for all $j\in\{1,\dots,d_{i}\}\setminus (C\cup D).$ Thus we can find $\mathcal{V}\subseteq\{1,\dots,d_{i}\}$ with $u_{d_{i}}(V)\geq 1-\sqrt{\kappa}-\sqrt{\varepsilon}$ and $\alpha(\phi(j))=\alpha(s\circ \phi_{t}(j))$ for $j\in \mathcal{V}.$ Thus
\begin{align*}
|\alpha^{d_{i}}(B_{\Delta_{2}}(s\circ \phi_{t},\varepsilon)\cap \Map_{\mu}(\Delta,F',\delta',L',\sigma_{i}))|&\leq \sum_{\substack{\mathcal{V}\subseteq\{1,\dots,d_{i}\},\\ |\mathcal{V}|\leq (\kappa+\sqrt{\varepsilon})d_{i}}}|A|^{|\mathcal{V}|}\\
&\leq \sum_{l=1}^{\lfloor{(\kappa+\sqrt{\varepsilon})d_{i}\rfloor}}\binom{d_{i}}{l}|A|^{l}.
\end{align*}
If $\kappa+\sqrt{\varepsilon}<1/2$ then for all large $i$ we have
\[\binom{d_{i}}{l}\leq \binom{d_{i}}{\lfloor{\kappa+\sqrt{\varepsilon\rfloor}}d_{i}}.\]
So by Stirling's Formula the above sum is at most
\[R(\kappa+\sqrt{\varepsilon})d_{i}\exp(d_{i}H(\kappa+\sqrt{\varepsilon}))|A|^{\kappa d_{i}}\]
for some constant $R>0,$ where
\[H(t)=-t\log t-(1-t)\log(1-t)\mbox{ for $0\leq t\leq 1$}.\]
Thus
\[h_{\Sigma,\mu}(\alpha)\leq h_{\Sigma,\mu}(\alpha;\beta,F,\delta,\sigma_{i})\leq H(\kappa+\sqrt{\varepsilon})+\kappa\log |A|+h_{\Sigma,\mu}(\Delta,\varepsilon/2,F',\delta',L').\]
Taking the infimum over all $F',\delta',L'$ and let $\kappa\to 0$ we have
\[h_{\Sigma,\mu}(\alpha)\leq h_{\Sigma,\mu}(\Delta,\varepsilon/2)+H(\sqrt{\varepsilon}).\]
Letting $\varepsilon\to 0$ and then taking the supremum over all $\alpha$ we have
\[h_{\Sigma,\mu}(X,\Gamma)\leq h_{\Sigma,\mu}(\Delta).\]

\end{proof}

\section{Spectral Consequences of Positive Entropy}\label{S:spectral}

	Let $\Gamma\actson (X,\mu)$ be a probability measure-preserving action of a countable discrete group. Associated to this action we have a natural representation
\[\rho_{\Gamma\actson (X,\mu)}\colon\Gamma\to U(L^{2}(X,\mu))\]
by
\[(\rho_{\Gamma\actson (X,\mu)}(g)f)(x)=f(g^{-1}x).\]
The space $\CC1$ inside $L^{2}(X,\mu)$ is clearly $\Gamma$-invariant, so we can consider the representation $\rho^{0}_{\Gamma\actson (X,\mu)}$ obtained by restricting $\rho_{\Gamma\actson (X,\mu)}$ to $L^{2}(X,\mu)\ominus \CC1.$ The representation  $\rho^{0}_{\Gamma\actson(X,\mu)}$ is called the \emph{Koopman} representation. Properties of a probability measure-preserving action are called \emph{spectral} when they only depend upon the Koopman representation. In this section, we deduce spectral properties of an action from assumptions of positive entropy.

\subsection{Representation Theoretic Preliminaries}

 We will need to apply the theory of representations of $*$-algebras. For this paper, we will only need unitary representations of groups, but later work will need this generality. For notation, if $A$ is a $*$-algebra and $\rho\colon A\to B(\mathcal{H})$ is a  $*$-representation and $E$ is a set we use $\rho^{\oplus E}$ for the $*$-representation of $A$ on $\ell^{2}(E,\mathcal{H})$ given by
\[(\rho^{\oplus E}(a)\xi)(x)=\rho(a)\xi(x).\]
Let us mention how the theory for $*$-algebras generalizes that of groups. If $\Gamma$ is a countable discrete group and   $\rho\colon\Gamma\to \mathcal{U}(\mathcal{H})$ is a unitary representation, for $f=\sum_{g\in\Gamma}f_{g}g\in \CC(\Gamma)$ we define
\[\rho(f)=\sum_{g\in\Gamma}f_{g}\rho(g).\]
We use the conjugate linear, antimultiplicative map $*$ on $\CC(\Gamma)$ given by
\[f^{*}=\sum_{g\in\Gamma}\overline{f_{g^{-1}}}g.\]
Under these operations $\CC(\Gamma)$ is $*$-algebra and $\rho(f^{*})=\rho(f)^{*}$ for any unitary representation of $\Gamma.$

If  $A$ is a $*$-algebra and $\rho_{j}\colon \Gamma\to \mathcal{U}(\mathcal{H}_{j})$ are two $*$-representations we write
\[\Hom_{A}(\rho_{1},\rho_{2})\]
for the space of bounded linear $A$-equivariant maps $\mathcal{H}_{1}\to \mathcal{H}_{2}.$

\begin{definition}\emph{ Let $A$ be a $*$-algebra and  $\rho_{j}\colon A \to B(\mathcal{H}_{j}), j=1,2$ be $*$-representations. We say that $\rho_{1}$ and $\rho_{2}$ are} mutually singular, \emph{written $\rho_{1}\perp \rho_{2}$  if for every pair of nonzero subrepresentations $\rho_{j}'$ of $\rho_{j}$ we have that $\rho'_{1}$ is not isomorphic to $\rho_{2}'.$ We say that $\rho_{1}$ is absolutely continuous with respect to $\rho_{2}$, and write $\rho_{1}\ll \rho_{2}$ if $\rho_{1}$ is embeddable in $\rho_{2}^{\oplus E}$ for some set $E.$}\end{definition}

The terminology is motivated by measure theory. For intuition, suppose that $A=C(X)$  and that $\rho_{j}\colon A\to B(\mathcal{H}_{j}),j=1,2$ are $*$-homomorphisms. Then we can find  spectral measures (in the sense of \cite{Conway} IX.1) $E_{j}$ on $X$ so that
\[\rho_{j}(f)=\int_{X}f(x)\,dE_{j}(x).\]
We leave it as an exercise to the reader to check that $\rho_{1}\perp\rho_{2}$ if and only if $E_{1}\perp E_{2},$ and similarly that $\rho_{1}\ll \rho_{2}$ if and only if $E_{1}\ll E_{2}$ (the definitions of absolute continuity and singularity of spectral measures is the same as for usual measures).

We will need the following equivalent conditions on singularity of representations. The following must be well known, but we include a proof for completeness. Throughout the proof, we shall use functional calculus. See \cite{Conway} Chapter VII, IX for background on functional calculus.
\begin{proposition}\label{P:singular} Let $A$ be a unital $*$-algebra  and $\rho_{j}\colon A\to B(\mathcal{H}_{j})$ be two unitary representations and suppose that $\mathcal{H}_{1},\mathcal{H}_{2}$ are separable. The following are equivalent:

(i): $\rho_{1}\perp \rho_{2}$

(ii): $\Hom_{A}(\rho_{1},\rho_{2})=\{0\}$

(iii): $\Hom_{A}(\rho_{2},\rho_{1})=\{0\}$

(iv): There is a sequence $x_{n}\in A$ so that $\max(\|\rho_{1}(x_{n})\|,\|\rho_{2}(x_{n})\|)\leq 1$ and
\[\rho_{1}(x_{n}^{*}x_{n})\to \id_{\mathcal{H}_{1}}\mbox{ in the strong operator topology}\]
\[\rho_{2}(x_{n}^{*}x_{n})\to 0\mbox{ in the strong operator topology}.\]
\end{proposition}

\begin{proof}
 The equivalence of $(ii)$ and $(iii)$ is proved by taking adjoints. To prove that $(ii)$ implies $(i)$ suppose that $\mathcal{K}_{j},j=1,2$ are closed, $A$-invariant, linear subspaces of $\mathcal{H}_{j},j=1,2$ and that $\phi\colon \mathcal{K}_{1}\to \mathcal{K}_{2}$ is an isomorphism. Define $T\colon \mathcal{H}_{1}\to \mathcal{H}_{2}$ by
\[T(\xi)=\Phi(\proj_{\mathcal{K}_{1}}(\xi)).\]
By (ii) we know that $T=0$ which implies that $\mathcal{K}_{j},j=1,2$ are zero.

	To see that $(i)$ implies $(ii)$ suppose that $T\in \Hom_{A}(\rho_{1},\rho_{2}).$ Let $T=U|T|$ be the polar decomposition (see \cite{Conway} VIII.3.11) . The fact that $T$ is equivariant implies that $T^{*}T$ is equivariant, and hence that $|T|=(T^{*}T)^{1/2}$ is, by approximating the square root function by polynomials. Since
\[U=SOT-\lim_{\varepsilon\to 0}T(|T|+\varepsilon)^{-1}\]
we see that $U$ is equivariant. Thus $U$ gives an $A$-equivariant isomorphism $(\ker T)^{\perp}\to \overline{\im T}.$ Since $\rho_{1}\perp \rho_{2}$ we find that $(\ker T)^{\perp}=0,$ and hence that $T=0.$

To prove that $(iv)$ implies $(ii)$ let $T\in\Hom_{A}(\rho_{1},\rho_{2}).$ Let $x_{n}$ be as in $(iv).$ Then, for any $\xi\in \mathcal{H}_{1}$
\[T(\xi)=\lim_{n\to \infty}T(\rho_{1}(x_{n}^{*}x_{n})\xi)=\lim_{n\to \infty}\rho_{2}(x_{n}^{*}x_{n})T(\xi)=0,\]
so $T=0.$

Suppose that $(iii)$ and $(ii)$ hold, we wish to prove $(iv).$ Recall that if $\mathcal{H}$ is a Hilbert space and $E\subseteq B(\mathcal{H})$ then $E'$ denotes the commutant of $E,$ i.e. $E'=\{T\in B(\mathcal{H}):TS=ST\mbox{ for all $S\in E$}\}.$ Suppose that
\[T\in (\rho_{1}\oplus \rho_{2})(A)',\]
Then we can regard $T$ as a matrix
\[T=\begin{bmatrix}
T_{11}&T_{12}\\
T_{21}&T_{22}
\end{bmatrix}\]
where $T_{ij}\in B(\mathcal{H}_{i},\mathcal{H}_{j}).$ Since $T\in (\rho_{1}\oplus \rho_{2})(A)'$ we see that $T_{ij}\in \Hom_{A}(\rho_{i},\rho_{j}).$ Thus $T_{12},T_{21}$ are $0.$ We thus see that
\[\begin{bmatrix}
\id_{\mathcal{H}_{1}}&0\\
0&0\\
\end{bmatrix}\in (\rho_{1}\oplus \rho_{2})(A)''=\overline{(\rho_{1}\oplus \rho_{2})(A)}^{SOT}\]
where the last equality follows from von Neumann's Double Commutant Theorem. We now prove $(iv)$ by using Kaplansky's Density Theorem.

\end{proof}
We need an analogue of the Lebesgue decomposition.

\begin{proposition}\label{P:LebesgueDecomposition} Let $A$ be a unital $*$-algebra  and $\rho_{j}\colon A\to B(\mathcal{H}_{j}),j=1,2$ be two $*$-representations. Then
\[\rho_{1}=\rho_{1,s}\oplus \rho_{1,c}\]
where $\rho_{1,s}\perp \rho_{2},$ and $\rho_{1,c}\ll\rho_{2}.$

\end{proposition}

\begin{proof} By Zorn's Lemma, we can find a maximal family $(\mathcal{K}_{\beta})_{\beta\in B}$ of pairwise orthogonal, $A$-invariant, closed, linear subspaces of $\mathcal{H}_{1}$ so that $\rho_{1}\big|_{\mathcal{K}_{\beta}}$ embeds into $\rho_{2}.$ Let
\[\mathcal{H}_{1,s}=\mathcal{H}_{1}\cap\left[ \bigoplus_{\beta\in B}\mathcal{K}_{\beta}\right]^{\perp}.\]
By maximality $\rho_{1}\big|_{\mathcal{H}_{1,s}}$ is singular with respect to $\rho_{2}.$ Setting
\[\mathcal{H}_{1,c}=\bigoplus_{\beta\in B}\mathcal{K}_{\beta}\]
and defining $\rho_{1,s},\rho_{1,c}$ by restricting $\rho_{1}$ to $\mathcal{H}_{1,s},\mathcal{H}_{1,c}$ completes the proof.

\end{proof}

\subsection{Proofs of the Main Applications}

\begin{theorem}\label{T:SingularLebesgue} Let $\Gamma$ be a countable discrete sofic group with sofic approximation $\Sigma.$ Let $(X,\mathcal{M},\mu)$ be a standard probability space with $\Gamma\actson (X,\mathcal{M},\mu)$ by measure-preserving transformations. Let $\mathcal{H}\subseteq L^{2}(X,\mu)$ be a closed, linear, $\Gamma$-invariant subspace such that  $\mathcal{M}$ is generated by
\[\{gf^{-1}(A):g\in\Gamma,f\in \mathcal{H}, A\subseteq \CC\mbox{ is Borel}\}.\]
 Suppose that
\[\rho_{\Gamma\actson(X,\mu)}\big|_{\mathcal{H}}\perp \lambda_{\Gamma}.\]
Then
\[h_{\Sigma,\mu}(X,\Gamma)\leq 0.\]

\end{theorem}

\begin{proof} Let $(f_{n})_{n=1}^{\infty}$ be a $\|\cdot\|_{2}$-dense subset of $\mathcal{H}.$ Let
\[\Phi\colon X\to \CC^{\NN\times\Gamma}\]
be defined by
\[\Phi(x)(n,g)=f_{n}(g^{-1}x)\]
and let
\[\nu=\Phi_{*}(\mu).\]
Let $\Gamma\actson \CC^{\NN\times\Gamma}$ by shifts. Since $\{gf_{n}^{-1}(A):g\in\Gamma,n\in\NN,A\subseteq\CC\mbox{ is Borel}\}$ generates $\mathcal{M}$ we see that $\Phi$ induces a $\Gamma$-invariant isomorphism $(X,\mu)\cong (\CC^{\NN\times\Gamma},\nu).$ Thus
\[h_{\Sigma,\mu}(X,\Gamma)=h_{\Sigma,\nu}(\CC^{\NN\times\Gamma},\Gamma).\]
 For $n\in \NN$ let  $Z_{n}\colon \CC^{\NN\times\Gamma}\to\CC$ be defined by
\[Z_{n}(y)=z(n,e),\]
then $Z_{n}\circ \Phi=f_{n}.$ Thus
\[\rho_{\Gamma\actson (\CC^{\NN\times\Gamma},\nu)}\big|_{\overline{\Span\{gZ_{n}:g\in\Gamma,n\in\NN\}}}\cong \rho_{\Gamma\actson (X,\mu)}\big|_{\mathcal{H}}\perp \lambda_{\Gamma}.\]
To simplify notation, we will use $\rho$ for $\rho_{\Gamma\actson(\CC^{\NN\times\Gamma},\nu)}.$ Let $\Delta$ be the dynamically generating pseudometric on $\CC^{\NN\times\Gamma}$ defined by
\[\Delta(z,w)=\sum_{n=1}^{\infty}2^{-n}\frac{|z(n,e)-w(n,e)|}{1+|z(n,e)-w(n,e)|}.\]
As $\CC^{\NN\times\Gamma}$ is clearly Polish, we can use $\CC^{\NN\times\Gamma},\Delta,\nu$ to do our computation. Let $\varepsilon>0$ be arbitrary, and let $0<\eta<\varepsilon$ be arbitrary. Let $\delta>0$ be sufficiently small  and $F\subseteq \Gamma$ be a sufficiently large finite set which will depend upon $\varepsilon,\eta$  in a manner to be determined later.  Given $\phi\in\Map(\Delta,F,\delta,\sigma_{i})$ for $n\in\NN$ define $z_{\phi,n}\in\CC^{d_{i}}$ by
\[z_{\phi,n}(j)=\phi(j)(n)(e)\]
and define $z_{\phi}\in(\CC^{\NN})^{d_{i}}$ by
\[z_{\phi}(j)(n)=z_{\phi,n}(j).\]
Conversely, given $z\in(\CC^{\NN})^{d_{i}}$ define $\psi_{z}\colon\{1,\dots,d_{i}\}\to \CC^{\NN\times\Gamma}$ by
\[\psi_{z}(j)(n,g)=z(\sigma_{i}(g)^{-1}(j))(n)\]
 If $F,\delta$ are chosen carefully, then for all sufficiently large $i$ we have that
\[\Delta_{2}(\psi_{z_{\phi}},\phi)<\varepsilon.\]
Choose $N\in\NN$ so that
\[2^{-N}<\varepsilon.\]
Since
\[\Gamma\actson\overline{\Span\{gZ_{n}:g\in\Gamma,n\in\NN\}\}}\perp\lambda_{\Gamma},\]
by Proposition \ref{P:singular} we may find a $\alpha\in\CC(\Gamma)$ so that $\max(\|\lambda(\alpha^{*}\alpha)\|,\|\rho(\alpha^{*}\alpha)\|)\leq 1$ and
\[\|\rho(\alpha^{*}\alpha)Z_{n}-Z_{n}\|_{2}<\eta,\mbox{ $1\leq n\leq N$}\]
\[\|\lambda(\alpha^{*}\alpha)\delta_{e}\|_{2}<\eta.\]
Let $M>0$ be sufficiently large depending upon $\varepsilon$ in a manner to be determined later. We will assume that $M$ is large enough so that there exists $G\in C_{c}(\CC)$ with $\|G\|_{\infty}\leq M$ and $G(z)=z$ for $|z|\leq M$ and
\[\|G\circ Z_{n}-Z_{n}\|_{2}<\varepsilon\mbox{ for $1\leq n\leq N$}.\]
Note that $G$ may be chosen independent of $\eta.$
As
\[\|\rho(\alpha^{*}\alpha)\|\leq 1\]
we then have
\begin{equation}\label{E:ifinitaryeajaksdjgpaoidj}
\|\rho(\alpha^{*}\alpha)G\circ Z_{n}-G\circ Z_{n}\|_{2}<3\varepsilon,\mbox{ for $1\leq n\leq N$}.
\end{equation}
Let $L\subseteq C_{b}(\CC^{\NN\times \Gamma})$ be sufficiently large in a manner to be determined later. We will assume that
\[L\supseteq\{G\circ Z_{1},\dots,G\circ Z_{N}\}.\]
We will use $\|\cdot\|_{2}$ for the $\ell^{2}$ norm on $\{1,\dots,d_{i}\}$ with respect to the uniform probability measure. By (\ref{E:ifinitaryeajaksdjgpaoidj}) if  $F\subseteq\Gamma,L\subseteq C_{b}(\CC^{\NN\times \Gamma})$ are sufficiently large and $\delta>0$ is sufficiently small then
\[\|\sigma_{i}(\alpha)^{*}\sigma_{i}(\alpha)G\circ z_{\phi,n}-G\circ z_{\phi,n}\|_{2}<6\varepsilon\mbox{ for $1\leq n\leq N$},\]
for all $\phi\in\Map_{\mu}(\rho,F,L,\delta,\sigma_{i})$ (with the notational conventions introduced after Definition \ref{S:soficdefn}). Let $p=\chi_{[1-\sqrt{\varepsilon},1+\sqrt{\varepsilon}]}(\sigma_{i}(\alpha)^{*}\sigma_{i}(\alpha))$ (this expression should be interpreted in the sense of functional calculus) then for all $\phi\in\Map_{\mu}(\rho,F,L,\delta,\sigma_{i}),$
\begin{align*}
\|p G\circ z_{\phi,n}-G\circ z_{\phi,n}\|_{2}^{2}&=\|\chi_{(\sqrt{\varepsilon},\infty)}(|\sigma_{i}(\alpha)^{*}\sigma_{i}(\alpha)-1|)G\circ z_{\phi,n}\|_{2}^{2}\\
&=\ip{\chi_{(\sqrt{\varepsilon},\infty)}(|\sigma_{i}(\alpha)^{*}\sigma_{i}(\alpha)-1|)G\circ z_{\phi,n},G\circ z_{\phi,n}}\\
&\leq \frac{1}{\varepsilon}\ip{(\sigma_{i}(\alpha)^{*}\sigma_{i}(\alpha)-1)^{2}G\circ z_{\phi,n},G\circ z_{\phi,n}}\\
&=\frac{1}{\varepsilon}\|(\sigma_{i}(\alpha)^{*}\sigma_{i}(\alpha)-1)(G\circ z_{\phi,n})\|_{2}^{2}\\
&<36\varepsilon
\end{align*}
and
\[\tr(p)\leq\frac{1}{1-\sqrt{\varepsilon}}\tr(\sigma_{i}(\alpha)^{*}\sigma_{i}(\alpha)).\]
Since
\[\lim_{i\to \infty}\tr(\sigma_{i}(\alpha)^{*}\sigma_{i}(\alpha)^{2})=\tau(\lambda(\alpha)^{*}\lambda(\alpha))=\ip{\lambda(\alpha)^{*}\lambda(\alpha)\delta_{e},\delta_{e}}\leq \|\lambda(\alpha)^{*}\lambda(\alpha)\delta_{e}\|_{2}<\eta\]
we see that for all large $i$ we have
\[\tr(p)<2\eta\]
provided $\varepsilon<\frac{1}{(2015)!}.$ Since $p$ is an orthogonal projection, we know  by Lemma \ref{L:triviallemma} we may choose an $\varepsilon$-dense subset $S$ of $Mp\Ball(\ell^{2}(d_{i},u_{d_{i}}))$ with
\[|S|\leq \left(\frac{3M+\varepsilon}{\varepsilon}\right)^{4\eta d_{i}}.\]
If $L$ and $M$ are sufficiently large, then
\[C_{i}=\{1\leq j\leq d_{i}:|z_{\phi,n}(j)|\leq M\mbox{ for $1\leq n\leq N$}\}\]
has
\[|C_{i}|\geq (1-\varepsilon)d_{i}.\]
Define $w_{\phi}\in (\CC^{\NN})^{d_{i}}$ by
\[w_{\phi}(j)(n)=G(z_{\phi,n}(j)).\]
For these values of $F,\delta,L,M$ we have
\begin{align*}
\Delta_{2}(\phi,\psi_{w_{\phi}})^{2}&\leq \varepsilon+\sum_{j=1}^{d_{i}}\sum_{n=1}^{N}2^{-n}\frac{|z_{\phi}(n,j)-w_{\phi}(n,j)|}{1+|z_{\phi}(n,j)-w_{\phi}(n,j)|}\\
&\leq 2\varepsilon
\end{align*}
as $w_{\phi}(n,j)=z_{\phi}(n,j)$ for $j\in\{1,\dots,d_{i}\}\setminus C_{i}.$ For $1\leq n\leq N$ choose $\xi_{\phi,n}\in S$ so that
\[\|pG\circ z_{\phi,n}-\xi_{\phi,n}\|_{2}<\varepsilon\]
and define $\xi_{\phi}\in (\CC^{\NN})^{d_{i}}$ by
\[\xi_{\phi}(j)(n)=\chi_{\{1,\dots,N\}}(n)\xi_{\phi,n}(j).\]
Then
\begin{align*}
\Delta_{2}(\psi_{w_{\phi}},\psi_{\xi_{\phi}})^{2}&\leq \varepsilon+\sum_{n=1}^{N}\frac{2^{-n}}{d_{i}}\sum_{j=1}^{d_{i}}|G(z_{\phi,n}(j))-\xi_{\phi,n}(j)|\\
&\leq \varepsilon+\max_{1\leq n\leq N}\|G\circ z_{\phi,n}-\xi_{\phi,n}\|_{2}\\
&\leq 2\varepsilon+6\sqrt{\varepsilon}
\end{align*}
the second line following from the Cauchy-Schwartz inequality.

Thus
\[\Delta_{2}(\phi,\psi_{\xi_{\phi}})\leq \sqrt{2}\sqrt{\varepsilon}+(2\varepsilon+6\sqrt{\varepsilon})^{1/2},\]

Hence for all large $i,$
\[S_{2\sqrt{2}\sqrt{\varepsilon}+2(2\varepsilon+6\sqrt{\varepsilon})^{1/2}}(\Map_{\mu}(\Delta,F,\delta,L,\sigma_{i}),\Delta_{2})\leq |S|^{N}\leq \left(\frac{3M+\varepsilon}{\varepsilon}\right)^{4N\eta d_{i}}.\]
Thus
\[h_{\Sigma,\mu}(\Delta,4\sqrt{2}\sqrt{\varepsilon}+4(\varepsilon+6\sqrt{\varepsilon})^{1/2})\leq 4 N\eta\log\left(\frac{3M+\varepsilon}{\varepsilon}\right).\]
Note that $\eta$ can be any number in $(0,\varepsilon).$ Thus we can let $\eta\to 0$ to find that
\[h_{\Sigma,\mu}(\Delta,4\sqrt{2}\sqrt{\varepsilon}+4(\varepsilon+6\sqrt{\varepsilon})^{1/2})\leq 0\]
since $\varepsilon>0$ is arbitrary we find that
\[h_{\Sigma,\mu}(X,\Gamma)\leq 0.\]

\end{proof}

The space $\mathcal{H}$ can be much smaller than $L^{2}(X,\mu).$ For example consider the case that $(X,\mu)=(B,\eta)^{\Gamma}$ and the action is Bernoulli. For $f\in L^{2}(B,\eta)$ and $g\in\Gamma,$ let $f_{g}\in L^{2}(X,\mu)$ be defined by
\[f_{g}(x)=f(x(g)).\]
Then we can take
\[\mathcal{H}=\overline{\Span\{f_{g}:g\in\Gamma,f\in L^{2}(B,\eta)\}}\]
and one can show that
\[\Gamma\actson L^{2}(X,\mu)\ominus \mathcal{H}\cong \ell^{2}(\Gamma)^{\oplus \infty}.\]
So indeed the space $\mathcal{H}$ is much smaller than $L^{2}(X,\mu).$  

%

For the next application, recall that the weak topology on $\Aut(X,\mu)$ is defined by saying that a basic neighborhood of $\alpha$ is given by $U_{A_{1},\dots,A_{n},\varepsilon}(\alpha)$ where $A_{1},\dots,A_{n}$ are measurable subsets of $X$ and $\varepsilon>0$ and
\[U_{A_{1},\dots,A_{n},\varepsilon}(\alpha)=\bigcap_{j=1}^{n}\{\beta\in\Aut(X,\mu):|\mu(\beta^{-1}(A_{j})\Delta\alpha^{-1}(A_{j}))|<\varepsilon\}.\]
 An action $\Gamma\actson (X,\mu)$ is \emph{compact} if there is a compact subgroup $K\subseteq \Aut(X,\mu)$ (for the weak topology) and a homomorphism $\pi\colon \Gamma\to K$ so that
\[gx=\pi(g)(x)\]
for almost every $x\in X.$

\begin{cor}\label{C:compact} Let $\Gamma$ be a countably infinite discrete sofic group with sofic approximation $\Sigma.$ Suppose that $\Gamma\actson (X,\mu)$ is a compact action. Then $h_{\Sigma,\mu}(X,\Gamma)\leq 0.$ \end{cor}
\begin{proof}

	Recall that a unitary representation $\rho\colon \Gamma\to U(\mathcal{H})$ is called \emph{weakly mixing} if
\[0\in\overline{\rho(\Gamma)}^{WOT},\]
and is \emph{compact} if
\[\overline{\rho(\Gamma)}^{WOT}\subseteq U(\mathcal{H}).\]
It is clear that a compact representation has no nontrivial weakly mixing subrepresentations. It is also well-known that the left regular representation is weakly mixing. Thus if $\Gamma\actson (X,\mu)$ is compact, then
\[\rho^{0}_{\Gamma\actson (X,\mu)}\perp \lambda_{\Gamma}\]
and we may apply Theorem \ref{T:SingularLebesgue}.

\end{proof}

In particular, note that if $K$ is a compact group and $\phi\colon \Gamma\to K$ is a  homomorphism, then the action $\alpha$ of $\Gamma$ on $K$ given by
\[\alpha(g)(x)=\phi(g)\cdot x\]
has entropy at most zero with respect to any sofic approximation.

\begin{definition}\emph{Let $\Gamma$ be a countable discrete sofic group with sofic approximation $\Sigma.$ We say that a probability measure-preserving action $\Gamma\actson (X,\mu)$ has} completely positive entropy \emph{(with respect to $\Sigma$) if whenever $\Gamma\actson (Y,\nu)$ is a factor of $\Gamma\actson (X,\mu)$ and  $Y$ is not a one-atom space then $h_{\Sigma,\nu}(Y,\Gamma)>0.$}
\end{definition}

\begin{cor}\label{C:cpe} Let $\Gamma$ be a countable discrete sofic  group with sofic approximation $\Sigma.$ Suppose that $\Gamma\actson (X,\mu)$ is a probability measure-preserving action which has completely positive entropy with respect to $\Sigma.$ Then $\rho^{0}_{\Gamma\actson (X,\mu)}\ll\lambda_{\Gamma}.$ \end{cor}

\begin{proof} By Proposition \ref{P:LebesgueDecomposition}, we can write
\[L^{2}(X,\mu)=\mathcal{H}_{1}\oplus \mathcal{H}_{2}\]
where $\rho_{\Gamma\actson (X,\mu)}\big|_{\mathcal{H}_{1}}\perp\lambda_{\Gamma}$ and
\[\rho_{\Gamma\actson (X,\mu)}\big|_{\mathcal{H}_{2}}\ll \lambda_{\Gamma}.\]
Suppose that $f\in \mathcal{H}_{1}.$ Define
\[\Phi\colon X\to \CC^{\Gamma}\]
by
\[\Phi(x)(g)=f(g^{-1}x)\]
and let $\nu=\Phi_{*}\mu,$ and let $\Gamma\actson \CC^{\Gamma}$ be the Bernoulli action. Then $\Gamma\actson (\CC^{\Gamma},\nu)$ is a factor of $\Gamma\actson (X,\mu).$ Set
\[\mathcal{K}=\overline{\Span\{g(f\circ \Phi):g\in \Gamma\}}\]
then
\[\{g\xi^{-1}(A):\xi\in\mathcal{K}, \mbox{Borel} A\subseteq \CC, g\in\Gamma\}\]
generates the Borel subsets of $\CC^{\Gamma}$ up to $\nu$-measure zero. Tautologically,
\[\mathcal{K}\cong \overline{\Span\{gf:g\in\Gamma\}}\perp\lambda_{\Gamma}.\]
Hence by Theorem \ref{T:SingularLebesgue} we know that
\[h_{\Sigma,\nu}(\CC^{\Gamma},\Gamma)\leq 0.\]
Since $\Gamma\actson (X,\mu)$ has completely positive entropy this implies that $(\CC^{\Gamma},\Gamma)$ is a one-atom space. But this is only possible if $f$ is constant. Thus $\mathcal{H}_{1}=\CC1$ and
\[\rho^{0}_{\Gamma\actson (X,\mu)}=\rho\big|_{\mathcal{H}_{2}}\ll \lambda_{\Gamma}.\]

\end{proof}

Again the above corollary illustrates the utility in not assuming that $\mathcal{H}=L^{2}(X,\mu)$ in Theorem \ref{T:SingularLebesgue} but instead just assuming that $\mathcal{H}$ generates $(X,\mu).$ If $(\CC^{\Gamma},\nu)$ is as in the above proof we do not a priori know that
\[\Gamma\actson L^{2}(\CC^{\Gamma},\nu)\ominus \CC1 \perp\lambda_{\Gamma}.\]

\end{document}